\documentclass[10pt]{amsart}

\usepackage{amsmath}
\usepackage{amssymb}
\usepackage{graphicx}
\usepackage{xcolor}
\usepackage{ytableau}

\usepackage{url}

\usepackage{amsfonts}

\usepackage{hyperref}

\RequirePackage{cleveref}
\usepackage{hypcap}
\hypersetup{colorlinks=true, citecolor=darkblue, linkcolor=darkblue}
\definecolor{darkblue}{rgb}{0.0,0,0.7}
\newcommand{\darkblue}{\color{darkblue}}
\newcommand{\deff}[1]{\emph{\darkblue #1}}

\usepackage{enumerate}

\newcommand{\excise}[1]{}

\numberwithin{equation}{section}

\newtheorem{thm}{Theorem}[section]
\newtheorem{lemma}[thm]{Lemma}

\newtheorem{cor}[thm]{Corollary}
\newtheorem{prop}[thm]{Proposition}

\newtheorem{conv}[thm]{Convention}

\newtheorem{ex}[thm]{Example}
\newtheorem{rem}[thm]{Remark}

\newtheorem{DT}[thm]{Definition/Theorem}
\newtheorem{Warn}[thm]{Caution}

    {\end{Example}}
    {\end{Remark}}

\def\wh{\widehat}

\def\emp{\varnothing}

\def\nn{\mathbb N}
\def\cc{\mathbb C}

\def\qqq{\mathbb Q}

\def\SS{\mathbb S}

\def\Ga{\Gamma}

\def\la{\lambda}
\def\ga{\gamma}

\def\de{\delta}

\def\be{\beta}

\def\SD{\textsc{D}}

\def\ssu{\subset}

\def\<{\langle}
\def\>{\rangle}

\def\rR{ {\text {\rm R} } }

\def\Ups{\Upsilon}
\def\ups{{a}}

\def\0{{\mathbf 0}}

\def\.{\hskip.06cm}
\def\ts{\hskip.03cm}

\newcommand{\DD}{\Ups}
\newcommand{\nent}{\operatorname{ne}}
\newcommand{\Ess}{\operatorname{Ess}}
\newcommand{\Dyck}{\mathsf{Dyck}}
\newcommand{\NDyck}{\mathsf{NDyck}}
\newcommand{\HP}{\mathcal{HP}}
\newcommand{\hp}{\operatorname{hp}}
\newcommand{\dd}{\ga}

\newcommand{\ED}{\mathcal{E}}
\newcommand{\EDLP}{\Pi}
\newcommand{\DDLP}{{\Delta}}

\newcommand{\EDS}{\mathcal{D}}

\newcommand{\RPP}{\operatorname{RPP}}
\newcommand{\SSYT}{\operatorname{SSYT}}
\newcommand{\SSVT}{\operatorname{SSVT}}
\newcommand{\SYT}{\operatorname{SYT}}
\newcommand{\SIT}{\operatorname{SIT}}
\newcommand{\IT}{\operatorname{IT}}
\newcommand{\BSIT}{\operatorname{BSYT}}
\newcommand{\Flag}{\operatorname{Flag}}

\def\SP{{\textup{\textsf{\#P}}}}

\def\.{\hskip.06cm}
\def\ts{\hskip.03cm}

\def\nin{\noindent}

\def\ms{s}
\def\pe{\pi}

\begin{document}

\title[New hook formulas for straight and skew shapes]{Hook formulas for skew shapes IV. Increasing tableaux and
  factorial Grothendieck polynomials}

\author[Alejandro Morales, Igor Pak, Greta Panova]{Alejandro H.~Morales$^\star$,
\ \ Igor Pak$^\diamond$ \ and \ \ Greta Panova$^\dagger$}

\thanks{\today}
\thanks{\thinspace ${\hspace{-.45ex}}^\star$Department of Mathematics and Statistics, UMass, Amherst, MA~01003.
\hskip.06cm
\texttt{ahmorales@math.umass.edu}}

\thanks{\thinspace ${\hspace{-.45ex}}^\diamond$Department of Mathematics,
UCLA, Los Angeles, CA~90095.
\hskip.06cm
Email:
\hskip.06cm
\texttt{pak@math.ucla.edu}}

\thanks{\thinspace ${\hspace{-.45ex}}^\dagger$Department of Mathematics,
 USC, Los Angeles, CA~90089.
\hskip.06cm
Email:
\hskip.06cm
\texttt{gpanova@usc.edu}}

\maketitle

\begin{abstract}
We present a new family of hook-length formulas for the number of
\emph{standard increasing tableaux} which arise in the study of
factorial Grothendieck polynomials.  In the case of straight shapes
our formulas generalize the classical \emph{hook-length formula}
and the \emph{Littlewood formula}.  For skew shapes, our formulas generalize
the \emph{Naruse hook-length formula} and its $q$-analogues, which
were studied in previous papers of the series.
\end{abstract}

\vskip.8cm

\section{Introduction}

\subsection{Foreword}
There is more than one way to explain a miracle.  First, one can show how it
is made, a step-by-step guide to perform it.  This is the most common
yet the least satisfactory approach as it takes away the joy and gives you
nothing in return.   Second, one can investigate away every
consequence and implication, showing that what appears to be miraculous
is actually both reasonable and expected.  This takes nothing away from the
the miracle except for its shining power, and puts it in the natural order of things.
Finally, there is a way to place the apparent miracle as a part of the general
scheme. Even, or especially, if this scheme is technical and unglamorous,
the underlying pattern emerges with the utmost clarity.

The \emph{hook-length formula} (HLF) is long thought to be a minor miracle,
a product formula for the number of certain planar combinatorial arrangements,
which emerges where one would expect only a determinant
formula.  Despite its numerous proofs and generalizations, including some
by the authors (see~$\S$\ref{ss:finrem-hist}), it continues to mystify and enthrall.
The goal of this paper is to give new curious generalizations of the HLF by using
\emph{Grothendieck polynomials}.  The resulting formulas are convoluted
enough to be unguessable yet retain the hook product structure to be instantly
recognizable.

\subsection{Straight shapes}\label{ss:intro-straight}
Recall some classical results in the area.
Let \ts $\la =(\la_1,\ldots,\la_\ell) \vdash n$ \ts be an \emph{integer partition}
of~$n$ with $\ell=\ell(\la)$ parts, and let \ts $f^\la := |\SYT(\la)|$
\ts be the number of \emph{standard Young tableaux} of shape~$\la$.  The \deff{hook-length formula}
by Frame--Robinson--Thrall~\cite{FRT} states that
\begin{equation} \label{eq:hlf} \tag{HLF}
f^{\lambda} \, = \, n!\, \prod_{u\in \lambda} \. \frac{1}{h(u)}\,,
\end{equation}
where \ts $h(u)=\lambda_i-i+\lambda'_j-j+1$ \ts is the \deff{hook-length} of the
square $u=(i,j)\in \la$.

Similarly, let \ts $\SSYT(\lambda)$ \ts denote the set of \emph{semi-standard Young tableaux}
of shape~$\la$.  For a tableau $T\in \SSYT(\lambda)$, let \ts $|T|$ \ts denote the sum of its entries.
The \deff{Littlewood formula}, a special case of the \emph{Stanley hook-content formula},
states that
\begin{equation} \label{eq:qhlf}\tag{$q$-HLF}
\sum_{T \in \SSYT(\lambda)} \. q^{|T|} \, = \, q^{b(\lambda)} \. \prod_{u\in \lambda} \. \frac{1}{1-q^{h(u)}}\,,
\end{equation}
where
$$b(\lambda) \, := \, \sum_{(i,j)\in \la} (i-1) \, = \, \sum_{i=1}^{\ell(\la)} \. (i-1)\ts \lambda_i\,,$$
see e.g.~\cite[$\S$7.21]{St2}.  Note that~\eqref{eq:qhlf} implies~\eqref{eq:hlf} by taking
limit $q\to 1$ and using a geometric argument, see~\cite[$\S$2]{Pak}, or the \emph{$P$-partition theory},
see~\cite[$\S$3.15]{St2}.
We are now ready to state the first two results of the paper, which
generalize~\eqref{eq:hlf} and~\eqref{eq:qhlf}, respectively.

\smallskip

For a tableau \ts $T\in \SSYT(\la)$, let \ts $T_k=\{u\in\la\.:\.T(u)=k\}$ \ts be the
set of tableau entries equal to~$k$.  Define \ts $T_{\le k}=\{u\in\la\.:\.T(u) \le k\}$ \ts, \ts $T_{\ge k}=\{u\in\la\.:\.T(u)\ge k\}$ \ts and $T_{< k} = T_{\leq k+1}$ similarly.  Finally, let \ts $\nu(T_{k})$,
\ts $\nu(T_{< k})$ \ts and \ts $\nu(T_{\ge k})$ \ts be the shapes of these tableaux.

We say that $T$ is a \deff{standard increasing tableau}
if it is strictly increasing in rows and columns, and $T_k$ is nonempty for all \ts $1\le k\le m$,
where \ts $m = m(T)$ \ts is the maximal entry in~$T$. Note that the (usual) standard Young tableaux
are exactly the standard increasing tableaux~$T$ with $m(T)=n$.
Denote by \ts $\SIT(\la)$ \ts the set
of standard increasing tableaux of shape~$\la$.  By definition, for \ts $T\in \SIT(\la)$, we have
\ts $0 \le \nu_i(T_{\le k})\le \la_i$ \ts is the number of elements in $T_{\le k}$ in $i$-th row of~$\la$.

\smallskip

\begin{thm} \label{t:khlf}
Fix $d\geq 1$. In the notation above, for every \ts $\la\vdash n$ with $\ell(\la)\leq d$,  we have:
\begin{equation} \label{eq:khlf}\tag{K-HLF}
\aligned
& \sum_{T \in \SIT(\lambda)} \,
\prod_{k=1}^{m(T)} \. \left(\left[\prod_{i=1}^d \.
  \frac{1+\beta\ts \bigl(\nu_i(T_{< k}\bigr)+d-i+1\big)}{1+\beta\ts\big(\lambda_i+d-i+1\big)}\right] -1\right)^{-1} \\
& \qquad = \ \frac{1}{(-\beta)^{n}} \, \.\prod_{i=1}^{\ell(\lambda)}
  \big(1+\beta(\lambda_i+d-i+1)\big)^{\lambda_i} \, \prod_{(i,j) \in \lambda} \. \frac{1}{h(i,j)}\,.
\endaligned
\end{equation}
\end{thm}

\smallskip

Here ``K'' in~\eqref{eq:khlf} stands for \emph{$K$-theory}, see below.  Note that~\eqref{eq:khlf}
implies~\eqref{eq:hlf} by taking the limit \ts $\be\to 0$, see Proposition~\ref{p:KHLF-HLF}.

\smallskip

To state the $K$-theory analogue of~\eqref{eq:qhlf}, we need a few more notation.
For a strictly increasing tableau \ts $T\in \SIT(\la)$, denote by $T_{\ge k}$
the skew subtableau of integers $\ge k$, and let \ts $\ups(T_{\ge k}):=|\nu(T_{\ge k})|$ \ts
denote the number of such integers.  This should not be confused with
\ts $|T_{\ge k}|$ \ts which is the sum of such integers.  Finally, denote
$$
\ms(\la) \, := \, \sum_{(i,j)\in \la} (i+j-1) \, = \, b(\la) + b(\la') + |\la| \..
$$

\smallskip

\begin{cor}\label{cor: q K-HLF}
In the notation above, for every \ts $\la\vdash n$,  we have:
\begin{equation} \label{eq:khlf-SIT}
\sum_{T \in \SIT(\lambda)} \. q^{|T|} \,
\prod_{k=1}^{m(T)} \. \frac{1}{1- q^{\ups(T_{\geq k})}}
\,\, = \,\, q^{\ms(\la)}
\prod_{(i,j) \in \lambda} \.\frac{1}{1-q^{h(i,j)}}\,.
\end{equation}
\end{cor}

\smallskip

The relationship between~\eqref{eq:khlf} and~\eqref{eq:khlf-SIT} is
somewhat indirect and both follow from a more general
equation~\eqref{eq:khlfq} by taking limits.

\smallskip

\begin{rem}\label{r:intro-RPP}{\rm
Denote by \ts $\RPP(\la)$ \ts the set of \emph{reverse plane partitions},
which are Young tableaux with entries $\ge 0$, weakly increasing in rows
and columns.  Similarly, denote by \ts $\IT(\la)$ \ts the set of
\emph{increasing tableaux}, which are Young tableaux with entries
$\ge 1$, strictly increasing in rows and columns.  Thus:
\begin{equation}
\label{eq:intro-inclusions}
\SYT(\la) \, \ssu \, \SIT(\la) \, \ssu \, \IT(\la) \,  \ssu \, \SSYT(\la) \, \ssu \,\RPP(\la)\ts.
\end{equation}
It is well known, and easily follows from~\eqref{eq:qhlf}, that
\begin{equation}\label{eq:intro-IT}
\sum_{T\in \IT(\la)} \. q^{|T|} \ = \ q^{\ms(\la)}
\sum_{T\in \RPP(\la)} \. q^{|T|} \ = \
q^{\ms(\la)}
\prod_{(i,j) \in \lambda} \. \frac{1}{1-q^{h(i,j)}}\,.
\end{equation}
Note that both~\eqref{eq:khlf-SIT} and~\eqref{eq:intro-IT} have
identical~RHS, but the~LHS of~\eqref{eq:khlf-SIT} has an extra product term.
In fact, there is a similar direct way to derive~\eqref{eq:khlf-SIT}
from~\eqref{eq:qhlf} by subtracting a constant to the entries in each anti-diagonal
of the tableau. However, this approach does not extend to skew
shapes, see Theorem~\ref{thm:skewKHLFqbetainfty} below and~$\S$\ref{ss:finrem-q-inf}.
}\end{rem}

\medskip

\subsection{Skew shapes} \label{ss:intro-skew}
We start with the \deff{Naruse hook-length formula} (NHLF), the subject of the
previous papers in this series~\cite{MPP1,MPP2,MPP3}.  Here we omit some
definitions; precise statements are given
in Section~\ref{s:excited}.

Let $\la/\mu$ be a \emph{skew Young diagram} (skew shape), and let \ts
$f^{\la/\mu}=|\SYT(\la/\mu)|$ \ts be the number of standard Young tableaux
of a shape~$\la/\mu$.  Then
\begin{equation} \label{eq:Naruse} \tag{NHLF}
f^{\lambda/\mu} \  = \ |\la/\mu|! \, \sum_{D \in \ED(\lambda/\mu)}\,
 \prod_{u \in \lambda\setminus D} \frac{1}{h(u)}\.\,,
\end{equation}
where \ts $h(u)$ \ts is the (usual) hook-length of square \ts $u\in \la$, and \ts
$\ED(\lambda/\mu)$ \ts denotes the set of {\em excited diagrams} of
shape~$\lambda/\mu$.
Note that when \ts $\mu = \emp$, there is a unique generalized excited diagram \ts $D=\emp$,
and~\eqref{eq:Naruse} reduces to~\eqref{eq:hlf}.

The $q$-analogue of~\eqref{eq:Naruse} generalizing Littlewood's
formula~\eqref{eq:qhlf} to skew shapes
was given by the authors in~\cite{MPP1}:
\begin{equation} \label{eq:qNHLF} \tag{$q$-NHLF}
\sum_{T \in \SSYT(\lambda/\mu)} \. q^{|T|} \ =
\ \sum_{D \in \mathcal{E}(\lambda/\mu)} \,\prod_{(i,j) \in \lambda\setminus D} \.
\frac{q^{\lambda'_j-i}}{1-q^{h(i,j)}}\,\..
\end{equation}
In Remark~\ref{r:intro-skew-SIT}, we discuss another notable $q$-analogue as a summation
over~$\RPP(\la/\mu)$.
The following results respectively generalize Theorem~\ref{t:khlf} and
Corollary~\ref{cor: q K-HLF}  to skew shapes, thus giving an advanced
generalizations of the~\eqref{eq:hlf}.

\smallskip

Let \ts $\mu\ssu\la$ \ts be two integer partitions.  Define the set \ts
$\SIT(\la/\mu)$ \ts of \deff{standard increasing tableaux of skew
shape}~$\la/\mu$ again as Young tableaux~$T$ which strictly increase in
rows and columns and have nonempty \ts $T_k$ \ts for all $1\le k \le m(T)$.
In this case, the \emph{generalized excited diagrams} were
introduced by Graham--Kreiman~\cite{GK} and Ikeda--Naruse~\cite{IN2}.
We denote the set of such diagrams by \ts $\EDS(\lambda/\mu)$,  and
postpone their definition until the next section.

\smallskip

\begin{thm}\label{thm: skew K-NHLF}
Fix $d\geq 1$. In the notation above, for every \ts $\mu \ssu \la$ with $\ell(\la)\leq d$,  we have:
\begin{equation}\label{eq:KNHLF-Skew} \tag{K-NHLF}
\aligned
&
\sum_{T \in \SIT(\lambda/\mu)}
\prod_{k=1}^{m(T)} \left(\left[\prod_{i=1}^d
  \frac{1+\beta\bigl(\nu_i(T_{< k})+d-i+1\bigr)}{1+\beta\ts (\lambda_i+d-i+1)}\right] -1\right)^{-1} \\ & \qquad \ = \
  \sum_{D\in \EDS(\lambda/\mu)} \, (-\beta)^{|D|\. - \.|\lambda|} \. \prod_{(i,j)\in
  \lambda \setminus D} \. \frac{\beta \ts (\lambda_i+d-i+1)\ts +\ts 1}{h(i,j)}\,.
\endaligned
\end{equation}
\end{thm}

\smallskip

See~$\S$\ref{ss:skew-OOF} for a completely different generalization of~\eqref{eq:hlf}
to skew shapes, which also has a $q$-analogue and $K$-theory analogue (Theorem~\ref{thm:skew K-OOF}).
Finally, Corollary~\ref{cor: q K-HLF} extends to skew shapes as follows:

\smallskip

\begin{thm}\label{thm:skewKHLFqbetainfty}
In the notation above, for every \ts $\mu \ssu \la$,  we have:
\begin{equation}\label{eq:skew-KHLF}
\sum_{T \in \SIT(\lambda/\mu)} \. q^{|T|} \,
\prod_{k=1}^{m(T)} \frac{1}{1-q^{\ups(T_{\geq k})}} \ \,= \ \sum_{D\in \EDS(\lambda/\mu)} \, \prod_{(i,j)\in
  \lambda \setminus D} \. \frac{ q^{h(i,j)}}{1 - q^{h(i,j)}} \,\,.
\end{equation}
\end{thm}

Again, equation \eqref{eq:skew-KHLF} reduces to~\eqref{eq:khlf-SIT} by taking
for \ts $\mu = \emp$,  and noting that
$$\sum_{(i,j)\in \la} \. h(i,j) \, = \, \sum_{(i,j)\in \la} \. (\la_j'-i+1) \, + \,
\sum_{(i,j)\in \la} \. (\la_i-j)  \, = \, \ms(\la)\ts.
$$

\smallskip

\begin{rem}\label{r:intro-skew-SIT}
{\rm While the inclusions in~\eqref{eq:intro-inclusions} continue to hold for skew shapes,
the natural analogue of~\eqref{eq:intro-IT} is no longer straightforward.  In fact,
for
$$I_{\la/\mu}(q) \. := \. \sum_{T\in \IT(\la/\mu)} \. q^{|T|} \, \quad \text{and} \, \quad
R_{\la/\mu}(q) \. := \. \sum_{T\in \RPP(\la/\mu)} \. q^{|T|}\,,
$$
the theory of P-partition gives:

\begin{equation} \label{eq:relation SIT and RPP skew case}
I_{\la/\mu}(-q) \. = \. q^{N} \. R_{\la/\mu}(1/q) \quad \text{for some \. $N\ge 0$,
\, see \, \cite[$\S$3.15]{St2}.}
\end{equation}

On the other hand, the summation formula for \ts $R_{\la/\mu}(q)$ \ts
given in~\cite[Thm.~1.5]{MPP1} gives yet another generalization of~\eqref{eq:Naruse},
but is summing over a different, albeit related, set of \emph{pleasant diagrams}
(see~$\S$\ref{ss:excited-back}):
\begin{equation}\label{eq:RPP-skew}
\sum_{T\in \RPP(\la/\mu)} \. q^{|T|} \ \,= \ \sum_{S\in \mathcal{P}(\lambda/\mu)} \,
\prod_{(i,j) \in S} \. \frac{q^{h(i,j)}}{1-q^{h(i,j)}}\,.
\end{equation}

As we explain in Section~\ref{s:skew}, equation~\eqref{eq:KNHLF-Skew} is really a generalization
of~\eqref{eq:RPP-skew} rather than~\eqref{eq:qNHLF}.  A connection can also ibe seen through
yet another summation formula for \ts $R_{\la/\mu}(q)$ \ts is given in \cite[Cor.~6.17]{MPP1}
in terms of (ordinary) excited diagrams and subsets \ts $\pe(\la/\mu)$ \ts of
\emph{excited peaks} (see the definition in~$\S$\ref{ss:excited-back}):
  \begin{equation}\label{eq:RPP-skew-excited-peaks}
\sum_{T\in \RPP(\la/\mu)} \. q^{|T|} \ \,= \ \sum_{D\in \ED(\lambda/\mu)} \,
q^{c(D)}\prod_{(i,j) \in \lambda\setminus D} \. \frac{1}{1-q^{h(i,j)}}\,,
\end{equation}
where \. $c(D) \ts := \ts \sum_{(i,j) \in \pe(\lambda/\mu)} \ts h(i,j)$.  Finally,
let us mention that the corresponding summation formula for \ts $I_{\la/\mu}(q)$ \ts
implied by~\eqref{eq:relation SIT and RPP skew case} and~\eqref{eq:RPP-skew-excited-peaks},
is obtained in~\eqref{eq:skewKHLF in terms of excited} more directly.
}
\end{rem}

\medskip

\subsection{Methodology}  \label{ss:intro-methodology}
While all results in this paper can be understood as enumeration of
certain Young tableaux, both the motivation and the proofs are algebraic.
This is routine in Algebraic Combinatorics, of course, and goes back to
the most basic and classical results in the area.

For example, for the LHS of~\eqref{eq:hlf}, we have \ts
$f^\la=\dim \SS^\la$, the dimension of the corresponding irreducible
$S_n$-module, with standard Young tableaux giving a natural basis.
On the other hand, the LHS in~\eqref{eq:qhlf} is equal to evaluation
of the Schur function \ts $s_\la(1,q,q^2,\ldots)$, and counts
multiplicities of \ts $\SS^\la$ \ts in the natural action on
the symmetric algebra \ts $\cc[x_1,\ldots,x_n]$ \ts graded by the degree.
The connection between the two are then provided by the combination of
Burnside and Chevalley theorems.

One can similarly define the standard Young tableaux of
skew shapes, excited diagrams, etc., even if the explanations become
more technical and involved with each generalization.  A tremendous
amount of work by many authors went into developments of this theory,
making  a proper overview for a paper of this scope impossible.
Instead, we skip to the end of the story and briefly describe
the motivation behind our new enumerative results.

Before we proceed to the recent work, it is worth pausing and pondering
on how the results in the area come about.  First, there are algebraic
areas (representation theory, enumerative algebraic geometry, etc.)
which provide the source of key algebraic objects (characters,
Schubert cells, characteristic classes, etc.)  Second, in order to
build the theory of these objects and be able to compute them,
combinatorial objects are extracted which are able to characterize
the algebraic objects (Schur functions, Schubert polynomials, etc.)

Third, the algebraic combinatorialists join the party and introduce the theory
of these combinatorial objects without regard to their algebraic origin.
Along the way they introduce a plethora of new combinatorial tools
(Young tableaux, reduced decompositions, RSK, etc.) which
substantially enhance and clarify the resulting combinatorial
structures. This is still the same theory, of course, but the
self-contained presentation and rich yet to be understood combinatorics
allows an easy access to people not algebraically inclined.

All this leads to the fourth wave, by enumerative combinatorialists
who are able to use tools and ideas from algebraic combinatorics
to study purely combinatorial problems.
This is where we find ourselves in this paper, staring with
an amazement at new enumerative results we obtain following
this course that we would not be able to dream up otherwise, yet
grasping for understanding of what these results really mean
in the grand scheme of things.

\medskip

\subsection{Motivation and background}  \label{ss:intro-motivation}
The main result of this paper is an unusual \deff{$\beta$-deformation} of
many known hook formulas.  Notably, our $\be$-deformation~\eqref{eq:khlf}
of the~\eqref{eq:hlf}, see Theorem~\ref{t:khlf}, remains concise and
multiplicative even if it is quite cumbersome at first glance.
By comparison, it is unlikely that \ts $g^\la:=|\SIT(\la)|$ \ts
has a closed formula (cf.~$\S$\ref{ss:finrem-CS}),
so a product formula for the \emph{weighted}
enumeration of $\SIT$s is both a minor miracle and testament to the
intricate nature of such tableaux.

The same pattern extends to other, more general hook formulas,
suggesting that~\eqref{eq:khlf} is not an accident, that the
$\beta$-deformation is a far-reaching generalization, on par
with the \emph{``$q$-analogue''}, \emph{``shifted analogue''}, etc.
We expect further results in this direction in the future.


In the combinatorial context, \emph{standard increasing tableaux}
(without the restriction on the values of the entries), appear as
byproducts of the classical {\em Edelman--Greene insertion}~\cite{EG,HY}
aimed at understanding of  Stanley's theorem on reduced factorizations
of {\em Grassmannian permutations} (permutations with at most one descent, see, e.g.~\cite{Man}).
They also appear in a more general setting
of the {\em Hecke insertion}~\cite{BuKSTY}.

More recently, standard increasing tableaux have appeared in the context
of $K$-theoretic version of the \emph{jeu de taquin}  of Thomas and
Yong~\cite{TY2,TY1}, and \emph{$K$-promotion} in
\emph{$K$-theoretic Schubert calculus}~\cite{Pec}.
Closely related \emph{semistandard set-valued tableaux}
were defined by Buch~\cite{Bu}, and have also been studied in a number
of papers.


In the algebraic context, the $K$-theory Schubert calculus of the Grassmannian
was introduced by Lascoux and Sch\"utzenberger~\cite{LS2}.  There, they defined
the \emph{Grothendieck polynomials} as representatives for $K$-theory classes
determined by structure sheaves of Schubert varieties.  The theory has been
rapidly developed in the past two decades.  We refer to~\cite{Bri,Buch}
for early surveys of the subject, as reviewing the extensive recent
literature is beyond the scope of this paper.


In this paper, the key role is played by the \emph{factorial Grothendieck
polynomials} \cite{McN,KMY}, which generalize both the well studied
\emph{Grothendieck polynomials} and \emph{factorial symmetric functions}.
The latter was first also introduced by Lascoux and Sch\"utzenberger~\cite{LS1}
in the guise of double Schubert polynomials for Grassmannian permutations,
and has been systematically studied by Macdonald~\cite{Mac},
see also~\cite{BMN} for further background.

Finally, let us mention the \emph{excited diagrams}, \emph{pleasant diagrams}
and the \emph{generalized excited diagrams}, which all arise in the context of
hook formulas of skew shapes, introduced by Ikeda--Naruse~\cite{IN1}, by us~\cite{MPP1},
and by Naruse--Okada~\cite{NO}, respectively.  These diagrams provide a combinatorial
language needed to state our results.

\medskip

\subsection{Proof ideas}  \label{ss:intro-proof}
For us, the story starts with our proof in~\cite{MPP1}
of equations~\eqref{eq:Naruse} and~\eqref{eq:qNHLF}
using evaluations of factorial Schur functions and the
{\em Chevalley type formulas}, see~\cite{MS}.  Naruse's
(unpublished) approach was likely similar, cf.~\cite{Strobl}.
After our paper, Naruse--Okada~\cite{NO} rederived and further
generalized to $d$-complete posets our $\RPP(\la/\mu)$
generalization \eqref{eq:RPP-skew-excited-peaks} of~\eqref{eq:Naruse}
using the Billey-type and the Chevalley-type formulas from the
\emph{equivariant $K$-theory}.  Note that our own proof of the
$\RPP(\la/\mu)$ summation~\eqref{eq:RPP-skew} given in~\cite{MPP1} is completely
combinatorial, and based on a generalization of the \emph{Hillman--Grassl
bijection}.

Our proofs in this paper combine our earlier proof technique in~\cite{MPP2}
with that of Naruse--Okada.  Namely, we study evaluations of the
factorial Grothendieck polynomials in two different ways.  First,
we use the Pieri rule for the factorial Grothendieck polynomials
to obtain the LHS of the equations in terms of increasing tableaux.
In the skew case, we combine these with the Chevalley type formulas.
We also use the Naruse--Okada
characterization of \emph{generalized excited diagrams} in terms of the usual
excited diagrams (see Proposition~\ref{prop: chara gen ED with EP}),
to obtain equation~\eqref{eq:skewKHLF} and its generalizations.
We also prove that these diagrams have a lattice path interpretation
that we exploit in~$\S$\ref{ss:excited-lattice} to obtain an upper
bound on their number.

Second, for the RHS of our hook formulas,
we use the vanishing property of the evaluation
for the case of straight shapes.  Finally, we use formulas
in terms of excited diagrams
of Graham--Kreiman~\cite{GK}  for the case of skew shapes.

\medskip

\subsection{Paper structure}  \label{ss:intro-structure}
We begin with preliminary Sections~\ref{s:Young} and~\ref{s:factorial},
where we review basic definitions and properties of permutation classes,
Young tableaux, increasing tableaux, and factorial Grothendieck polynomials.
We then proceed to present proofs of all our hook
formulas via more general multivariate formulas.

Namely, in Section~\ref{s:straight}, we prove Theorem~\ref{thm:khlf_multi},
the main result of the straight shape case, which implies
Theorem~\ref{t:khlf} and Corollary~\ref{cor: q K-HLF}.  In Section~\ref{s:excited}
we review the technology of excited diagrams that was unnecessary for the
straight shape.  We also relate our notation and results
to further clarify combinatorics of the double Grothendieck polynomials of vexillary permutations
for devotees of the subject. Then, in Section~\ref{s:skew}, we prove
Theorem~\ref{thm:skewKHLF multi},  the main and most general result
of this paper, which similarly implies both  Theorems~\ref{thm: skew K-NHLF}
iand~\ref{thm:skewKHLFqbetainfty}.

Let us emphasize that this paper is not self-contained by any measure,
as we are freely using results from the area and from our previous papers
in this series.  We tried, however, to include all necessary
definitions and results, so the paper can be read by itself.  This
governed the style of the paper: we covered the straight shape case
first as it requires less of a background and can be understood by
a wider audience.  This also helped set up the more general skew shape
case which followed.   We conclude with final remarks and open
problems in Section~\ref{s:finrem}.

\bigskip

\section{Permutations, Dyck paths and Young tableaux}\label{s:Young}

\smallskip

\subsection{Basic notation}\label{ss:Young-basic}
Let \ts $\nn = \{0,1,\ldots\}$ \ts and \ts $[n]=\{1,\ldots,n\}$.

\medskip

\subsection{Permutations}\label{ss:Young-perm}
We write permutations of $[n]$ as \ts $w=w_1w_2\ldots w_n \in S_n$, where $w_i$ is the image of~$i$.
The \deff{Rothe diagram} of a permutation $w$ is the subset of \ts $[n]\times [n]$ \ts given by \ts
$\rR(w) := \bigl\{ (i,w_j) \,\. | \, \, i<j, \ts w_i >w_j\bigr\}$. The \deff{essential set}
of a permutation~$w$ is the subset of $\rR(w)$ given by \ts
$\Ess(w) := \bigl\{ (i,j) \in R(w) \mid (i+1,j), (i,j+1), (i+1,j+1) \not\in \rR(w) \bigr\}$, see e.g.~\cite[\S 2.1-2]{Man}.

A permutation $w\in S_n$ is called \deff{Grassmannian} if it has a unique descent, say at position~$k$. Such a Grassmannian permutation corresponds to a partition $\mu=\mu(w)$ with $\ell(\mu)  \leq k$ and $\mu_1 \leq n-k$. Grassmannian permutations $w$ can also be characterized as having $\Ess(w)$ contained in one row, the last row of $\rR(w)$ and $\mu(w)$  can be read from the number of boxes of $\rR(w)$ in each row bottom to top.

A permutation $w\in S_n$ is called \deff{vexillary} if it is $2143$-avoiding.
Vexillary permutations can also be characterized as permutations $w$ where $\rR(w)$ is,
up to permuting rows and columns, the Young diagram of a partition \ts $\mu=\mu(w)$.
 Given a vexillary permutation let $\lambda=\lambda(w)$
be the smallest partition containing the diagram~$\rR(w)$. We call this partition the
\deff{supershape} of $w$ and note that \ts $\mu(w)\subseteq \lambda(w)$. The Young diagram of $\lambda(w)$ can also be obtained by taking the union over \ts $i \times j$ \ts rectangles with NW and SE corners $(1,1)$ and $(i,j)$ for each $(i,j)$ in $\Ess(w)$. Note also that Grassmannian permutations are examples of vexillary permutations.

\medskip

\subsection{Lattice paths}\label{ss:Young-paths}
A \deff{lattice path} contained in a Young diagram $\lambda$ is a path of steps $(1,0)$ and $(0,1)$ along the square grid centered at the centers of the cells of $\lambda$.

A \deff{Dyck path} $\dd$ of length $2n$ is a lattice path from $(0,0)$ to $(2n,0)$ with steps $(1,0)$ and $(1,-1)$ that stay on or above the $x$-axis. The set of Dyck paths of length $2n$ is denoted by $\Dyck(n)$. For a Dyck path $\dd$, a \deff{peak} is a point $(c,d)$ such that $(c-1,d-1)$ and $(c+1,d-1)$ are in~$\dd$. A peak $(c,d)$ is called \deff{high peak} if $d>1$. The set of high peaks of a Dyck path $\dd$ is denoted by \ts $\HP(\dd)$ \ts and its size by~$\hp(\dd)$. Note that a Dyck path, upon rotation and rescaling is also a lattice path in the Young diagram of $\delta_n=(n+1,n,\ldots,1)$.

For general lattice paths $\gamma$ above a certain base path $\gamma'$ we can also define \deff{ high peaks} relative to $\gamma'$ as the set of pints $(c,d)$, such that $(c,d-1),(c+1,d) \in \gamma$ and $(c,d) \not \in \gamma'$. We will also denote this set by $\HP(\gamma)$.

\medskip

\subsection{Plane partitions and Young tableaux} \label{ss:Young-def}
We use standard English notation for Young diagrams and Young tableaux,
see e.g.~\cite[$\S$7]{St2}.

To simplify the notation, we use the same letter to denote an
integer partition and the corresponding \deff{Young diagram}~$\la$.
The \deff{skew shape} (skew Young diagram)~$\la/\mu$ is given by a pair of Young
diagrams, such that \ts $\mu \subset \la$.  Denote \ts $|\la/\mu|$ \ts
the \deff{size} of the skew shape.

A \deff{reverse plane partition} of skew shape
$\lambda/\mu$ is an array \ts $A=(a_{ij})$ \ts of nonnegative integers of shape $\lambda/\mu$
that is weakly increasing in rows and columns.  A \deff{semistandard Young tableau} (SSYT)
of shape $\lambda/\mu$ is a reverse plane partition of shape $\lambda/\mu$ that is strictly
increasing in columns and has entries~$\ge 1$. We denote these sets of tableaux by
$\RPP(\lambda/\mu)$ and $\SSYT(\lambda/\mu)$,
respectively.

A \deff{standard Young tableau} of shape $\lambda/\mu$ is an
reverse plane partition $T$ of shape $\lambda/\mu$ which contains entries
\ts $1,\ldots,|\la/\mu|$ \ts exactly once.
We denote this set by \ts $\SYT(\lambda/\mu)$, and
let \ts $f^{\lambda/\mu}:=|\SYT(\lambda/\mu)|$.

In less standard notation, for a tableau \ts $T\in \RPP(\la)$,
we define tableaux \ts $T_k$, \ts $T_{\le k}$ \ts
and \ts $T_{\ge k}$ \ts as in the introduction.  The (skew) shape of a
tableau~$Q$ is denote by $\nu(Q)$.  We are using \ts $\ups(Q):=|\nu(Q)|$ \ts
to denote the size (the number of entries) in~$Q$.  As in the introduction,
we write \ts $|T|$ \ts to denote the sum of entries in the tableau~$T$.

\medskip

\subsection{Increasing and set-valued Young tableaux} \label{ss:Young-inc}
An \deff{increasing tableau} of shape  $\lambda/\mu$ is a
row strict semistandard Young tableau of shape $\lambda/\mu$.
A \deff{standard increasing tableau}\footnote{In the literature these tableaux
are sometimes called (just) \emph{increasing tableaux} or \emph{packed increasing
tableaux}~\cite{Pec2}.} is an increasing
tableau of shape $\lambda/\mu$ whose entries are exactly \ts $[m]$,
for some \ts $m \leq |\lambda/\mu|$. As in the introduction,
we denote by \ts $m(T):=m$ \ts the maximal entry in~$T$.

Denote by \ts $\IT(\lambda/\mu)$ \ts the set of increasing tableaux, and by \ts
$\SIT(\lambda/\mu)$ \ts the set of standard increasing tableaux of shape
$\lambda/\mu$.  Let \ts $g^{\lambda/\mu}:=|\SIT(\lambda/\mu)|$ \ts
be the number of standard increasing tableaux.

Tableau \ts $T\in \SIT(\la/\mu)$ \ts is called a
\deff{barely standard Young tableau} of shape $\lambda/\mu$, if $m(T)=|\la/\mu|-1$.
In other words, these are the standard increasing tableaux with exactly
one entry appearing twice (cf.~$\S$\ref{ss:finrem-BSIT}).  We
denote the set of these tableaux by $\BSIT(\lambda/\mu)$. We also
denote by $\BSIT_k(\lambda/\mu)$ the tableaux in $\BSIT(\lambda/\mu)$
with entry \ts $k$ \ts appearing twice.

Finally, a \deff{semistandard set-valued tableau} of shape $\lambda/\mu$ is an
assignment of subsets of \ts $[n]$ \ts to the cells
of $\lambda/\mu$, such that for $T(u)$ is the set in cell $u\in \la$, we have:

\qquad $\circ$ \
$\max T(u) \leq \min T(u')$, where $u'$ is the cell to the right of $u$ in the same row, and

\qquad $\circ$ \
$\max T(u) <\min T(u')$, where $u'$ is the cell below $u$ in the same column.

\nin
We use $\nent(T)$ to denote the \deff{number of entries} of~$T$,
and \ts $\SSVT_n(\la/\mu)$ \ts to denote the set of such tableaux.

\medskip

\subsection{Examples} \label{ss:Young-ex}
To illustrate the definitions, in the figure below we have \ts
$\la = 442$, \ts $\mu=21$, \ts $A\in \RPP(\la/\mu)$,
\ts $B\in \SSYT(\la/\mu)$, \ts $C\in \SYT(\la/\mu)$, \ts $D \in \SSVT_5(\la/\mu)$,
\ts $E \in \IT(\la/\mu)$, \ts $F \in \SIT(\la/\mu)$ \ts with \ts $m(F)=5$, and \ts
$G \in \BSIT_3(\la/\mu)$.  Note that \ts $\nent(D)=9$.

\smallskip

{\small

$$\aligned
A \, & = \
\ytableaushort{{*(green)}{*(green)}{*(yellow)0}{*(yellow)1},{*(green)}{*(yellow)0}{*(yellow)0}{*(yellow)1},{*(yellow)1}{*(yellow)2}} \ , \quad
B \, = \
\ytableaushort{{*(green)}{*(green)}{*(yellow)1}{*(yellow)1},{*(green)}{*(yellow)1}{*(yellow)2}{*(yellow)3},{*(yellow)3}{*(yellow)3}} \ , \quad
C \, = \
\ytableaushort{{*(green)}{*(green)}{*(yellow)2}{*(yellow)5},{*(green)}{*(yellow)1}{*(yellow)4}{*(yellow)6},{*(yellow)3}{*(yellow)7}} \ , \quad
D \, = \
\ytableaushort{{*(green)}{*(green)}{*(yellow)1}{*(yellow)1,4},{*(green)}{*(yellow)1}{*(yellow)3}{*(yellow)5},{*(yellow)1,2}{*(yellow)2}} \\
& \\
& \hskip1.1cm E \, = \
\ytableaushort{{*(green)}{*(green)}{*(yellow)1}{*(yellow)3},{*(green)}{*(yellow)1}{*(yellow)2}{*(yellow)4},{*(yellow)2}{*(yellow)7}} \ , \qquad
F \, = \
\ytableaushort{{*(green)}{*(green)}{*(yellow)1}{*(yellow)3},{*(green)}{*(yellow)1}{*(yellow)2}{*(yellow)4},{*(yellow)2}{*(yellow)5}} \ , \qquad
G \, = \
\ytableaushort{{*(green)}{*(green)}{*(yellow)2}{*(yellow)3},{*(green)}{*(yellow)1}{*(yellow)4}{*(yellow)5},{*(yellow)3}{*(yellow)6}}
\endaligned
$$
}
\smallskip

\nin
In this case, we have \ts $|F|=18$, \ts $\nu(F_{\le 0}) = \mu=21$, \ts $\nu(F_{\le 1}) = 32$, \ts $\nu(F_{\le 2}) = 331$,
\ts $\nu(F_{\le 3}) = 431$, \ts $\nu(F_{\le 4}) = 441$, and \ts $\nu(F_{\le 5}) = \la = 442$.  Similarly,
\ts $\nu(F_{\ge 2})=442/32$ \ts and \ts $\ups(F_{\ge 2}) = 5$.

\smallskip

Finally, in the notation of the introduction, we have \ts $b(\la)=|N_\la|$ \ts and
\ts $\ms(\la)=|M_\la|$ \ts is the sum of the entries of the minimal reverse
plane partition $N_\la\in \RPP(\la)$, and minimal strictly increasing
tableau \ts $M_\la \in \SIT(\la)$, with entries \ts
$N_\la(i,j)=(i-1)$ \ts and
$M_\la(i,j)= (i+j-1)$, respectively.  See an example in the figure below:

\smallskip
{\small
$$
N_{442} \, = \,\, \ytableaushort{{*(pink)0}{*(pink)0}{*(pink)0}{*(pink)0},{*(pink)1}{*(pink)1}{*(pink)1}{*(pink)1},{*(pink)2}{*(pink)2}}
\qquad \text{and} \qquad
M_{442} \, = \,\, \ytableaushort{{*(cyan)1}{*(cyan)2}{*(cyan)3}{*(cyan)4},{*(cyan)2}{*(cyan)3}{*(cyan)4}{*(cyan)5},{*(cyan)3}{*(cyan)4}}
$$
}
\smallskip

\nin
In this case \ts $b(\la) = |N_{442}|=8$ \ts and $\ms(\la)=|M_{442}|=31$.

\medskip

\subsection{Special cases} \label{ss:Young-sp}
To further clarify the definitions, let us give a quick calculation
of the number of increasing tableaux for the \deff{two row shape}
$(n,n)$ and the \deff{hook shape} $(p,1^q)$.

Let $s_n$ denote the $n$-th \deff{little Schr\"oder number}
\cite[\href{http://oeis.org/A001003}{A001003}]{OEIS} that counts lattice paths
\ts $(0,0)\to (n,n)$ \ts with steps $(1,0)$, $(0,1)$ and $(1,1)$ that never go below
the main diagonal $x=y$ and with no $(1,1)$ steps on the diagonal.

\smallskip

\begin{prop}[{\cite{Pec}}]
We have \ts $g^{(n,n)}=s_n$.
\end{prop}

\begin{proof}
 We interpret the SITs as lattice paths on the square grid. In the case $\la= (n,n)$,
 let $T\in \SIT(\la)$ correspond to the lattice path \ts $\ga:(0,0) \to (n,n)$ \ts is given
 by a sequence of steps:

 $(1,0)$ \, if the entry  \. $i$ \.  appears only in the first row of $T$,

 $(0,1)$ \, if the entry  \. $i$ \.  appears only in the second row of $T$, and

 $(1,1)$ \,  if the entry  \. $i$ \. appears on both rows.

\nin
The increasing columns condition forces the paths~$\ga$ not to cross below the diagonal,
with all $(1,1)$ steps strictly above the diagonal, as desired.
\end{proof}

\smallskip

Similarly, let $\SD(m,n)$ denoted the {\em Delannoy number}
\cite[\href{http://oeis.org/A008288}{A008288}]{OEIS}
that counts lattice paths \ts $(0,0) \to (m,n)$ \ts with steps \ts
$(0,1)$, $(1,0)$ and $(1,1)$. We call these \deff{Delannoy steps}.

\smallskip

\begin{prop}[{\rm cf.~\cite{PSV}}]
 \label{prop: SIT for hook shape}
For the hook shape $\lambda=(p,1^q)$, we have \ts $g^{(p,1^q)}=\SD(p-1,q)$.
\end{prop}

\smallskip

The proof follows verbatim the argument above, but the lattice paths
with Delannoy steps no longer have a diagonal constraint.  We omit the details.

\bigskip

\section{Factorial Grothendieck polynomials}\label{s:factorial}

Recall the following operators first introduced in \cite{FK2,FK1}:
\begin{align*}
x\oplus y \. &:= \. x+y+\beta xy\ts, \qquad
x \ominus y \. := \.
\frac{(x-y)}{(1+\beta y)}\., \qquad \ominus \. x \. := \. 0 \ominus x\ts, \\
& \qquad \text{and} \qquad [x\.|\.{\bf y}]^k \. := \. (x
\oplus y_1)\ts (x\oplus y_2) \. \cdots \. (x \oplus y_k)\ts ,
\end{align*}
where ${\bf y}=(y_1,y_2,\ldots)$.

\smallskip

\begin{DT}[{McNamara~\cite{McN}}] \label{defthm: fG}
\deff{Factorial Grothendieck polynomials} are defined by either of the following: 
\begin{align*}
G_{\mu}(x_1,\ldots,x_d\.|\.{\bf y}) \ :&= \ \sum_{T\in \SSVT_d(\mu)} \beta^{\nent(T)-|\mu|} \prod_{u\in \mu, \. r\in
  T(u)} \. \bigl(x_r \oplus y_{r+c(u)}\bigr)\\
&= \ \, \det\Bigl([x_i\.|\.{\bf y}]^{\mu_j+d-j}(1+\beta x_i)^{j-1}\Bigr)_{i,j=1}^d \, \prod_{1\leq i<j \leq d} \. \frac{1}{(x_i-x_j)}\,.
\end{align*}
\end{DT}

\smallskip

The factorial Grothendieck polynomials are equal to the \deff{double
  Grothendieck polynomials} parameterized by a Grassmannian permutation
associated to partition~$\mu$, see~\cite{McNadd}. These in turn were defined
earlier in~\cite{KMY}, in the greater generality of all vexillary permutations,
see equation~\eqref{eq: fG is dG} below.
We postpone their definition until~$\S$\ref{ss:excited-GP}
(see also~$\S$\ref{ss:finrem-factorial}).

\smallskip

\begin{rem}{\rm
 As mentioned in \cite[Rem.~3.2]{McN}, in the literature Grothendieck
polynomials sometimes appear only in the case $\beta=-1$. However, one can
obtain the $\beta$ case from the case $\beta=-1$ by replacing $x_i$
with $-x_i/\beta$ and $y_i$ with $y_i/\beta$,
\begin{equation} \label{eq:betaat-1}
G_{\mu}\bigl({\bf x}\.|\.{\bf y}\bigr) \. \big|_{\beta=-1} \ \, = \
(-\beta)^{|\mu|}\cdot G_{\mu}\bigl(-{\bf x}/\beta~|~{\bf y}/\beta\bigr).
\end{equation}
}\end{rem}

\smallskip

It is easy see that \ts $G_\emp({\bf x}\ts|\ts{\bf y}) =1$.
We need the following technical result.

\smallskip

\begin{prop}[{\cite{McN,McNadd}}] \label{prop:props of fG} The factorial Grothendieck polynomials \ts $G_\emp({\bf x}\ts|\ts{\bf y})$ \ts
satisfy:
\begin{enumerate}[(i)]
\item $G_{\mu}(x_1,\ldots,x_d\.|\.{\bf y})$ \. is symmetric in \ts $x_1,x_2,\ldots,x_d$\ts.
\smallskip
\item Doing the substitution \ts $y_i \gets (-y_i)$, and setting \ts $\beta= 0$,
we obtain the factorial Schur function:
\[
G_{\mu}\bigl(x_1,\ldots,x_d \.|\. -{\bf y}\bigr) \. \big|_{\beta=0} \ = \
s_{\mu}(x_1,\ldots,x_d\.|\. {\bf y}).
\]

\item Setting \ts $y_i=0$,
we obtain the ordinary Grothendieck polynomials:
\[
G_{\mu}(x_1,\ldots,x_d \mid {\bf y}) \. \big|_{y_i=0} \ = \ G_{\mu}(x_1,\ldots,x_d).
\]

\item They are equal to double Grothendieck polynomial of Grassmannian permutations:
\begin{equation} \label{eq: fG is dG}
\mathfrak{G}_{w(\mu)}({\bf x}\ts,\ts{\bf y}) \ = \
G_{\mu}(x_1,\ldots,x_d~|~{\bf y}),
\end{equation}
for
  $d\geq \ell(\mu)$, and $w(\mu)$ is the {\em Grassmannian
  permutation} with descent at position $d$ associated to $\mu$.
\end{enumerate}
\end{prop}

\smallskip

\begin{prop}[{\deff{Vanishing property of Grothendieck polynomial}~\cite[Thm.~4.4]{McN}}] \label{prop:Gvanprop}
When evaluated at \ts ${\bf y}_{\lambda}:=(\ominus \.
  y_{\lambda_1+d},\ominus \. y_{\lambda_2+d-1},\ldots,\ominus \. y_{\lambda_d+1})$ with $\ell(d) \leq d$,
\begin{equation}\label{eq:IC}
G_{\mu}({\bf y}_{\lambda}\.|\.{\bf y}) \ = \ \begin{cases}
0 &\text{ if } \mu\not\subseteq \lambda,\\
\prod_{(i,j)\in \lambda} (y_{d+j-\lambda'_j} \ominus
y_{\lambda_i+d-i+1}) &\text { if } \mu = \lambda.
\end{cases}
\end{equation}
\end{prop}

\smallskip

To simplify the notation, we write \ts $G_1$ for $G_{(1)}$. We use notation \ts
$\nu \mapsto \mu$ \ts when the skew shape $\nu/\mu$
is nonempty and its boxes are in different rows and columns.  Note that \ts $\nu\ne \mu$
\ts in this case.  In this notation, every standard increasing tableau
\ts $T \in \SIT(\lambda/\mu)$ \ts is viewed as a chain
\begin{equation}\label{eq:SIT-maps}
\la = \nu\bigl(T_{\le k}\bigr) \. \mapsto \. \nu\bigl(T_{\le k-1}\bigr)\mapsto \.
\ldots \. \mapsto \. \nu\bigl(T_{\le 1}\bigr) \. \mapsto \. \nu\bigl(T_{\le 0}\bigr) = \mu\ts.
\end{equation}

\smallskip

\begin{lemma}[{\deff{Pieri rule for Grothendieck polynomial}~{\cite[Prop.~4.8]{McN}}}]
\begin{equation} \label{eq:GrothPieri}
G_{\mu}({\bf x}\.|\.{\bf y})\bigl(1+\beta G_1({\bf x}\.|\.{\bf y})\bigr) \ = \
\bigl(1+\beta G_1({\bf y}_{\mu}\.|\.{\bf y})\bigr) \, \sum_{\nu\mapsto\mu} \.\beta^{|\nu/\mu|}
G_{\nu}({\bf x}\.|\.{\bf y})\..
\end{equation}
\end{lemma}

\smallskip

We can rewrite this Pieri rule as follows:

\smallskip

\begin{prop} We have:
\begin{equation} \label{eq:GrothPieri2}
G_{\mu}({\bf x}\.|\.{\bf y})\left(\frac{G_1({\bf x}\.|\.{\bf
      y})-G_1({\bf y}_{\mu}\.|\.{\bf y})}{1+\beta G_1({\bf
      y}_{\mu}\.|\.{\bf y})}
\right) \ = \ \sum_{\nu\mapsto\mu} \. \beta^{|\nu/\mu|-1} G_{\nu}({\bf x}\.|\.{\bf
  y})\..
\end{equation}
\end{prop}

\begin{proof}
We expand both sides of \eqref{eq:GrothPieri} and cancel the term
$G_{\mu}({\bf x}| {\bf y})$ giving
\[
G_{\mu}({\bf x}\.|\.{\bf y}) \ts \cdot\ts \beta \ts G_1( {\bf x}\.|\. {\bf y}) \ = \ \beta \ts
G_1({\bf y}_{\mu}\.|\. {\bf y}) \ts \cdot\ts G_{\mu}({\bf x} \.|\. {\bf y}) \. + \. \bigl(1+\beta
G_1({\bf y}_{\mu}\.|\.{\bf y})\bigr) \. \sum_{\nu\mapsto \mu} \beta^{|\nu/\mu|}
G_{\nu}({\bf x}\.|\.{\bf y})\..
\]
Now collect the terms with \ts $G_{\mu}({\bf x}\.|\.{\bf y})$ \ts on the LHS.  Dividing by
$1+\beta G_1({\bf y}_{\mu}\.|\. {\bf y}) \neq 0$ \ts and $\beta$ gives the desired expression.
\end{proof}

\begin{rem}{\rm
When we set $\beta=0$ in the Pieri rule above, it immediately reduces
to the Pieri rule of factorial Schur functions (see e.g.~\cite[\S 3]{MS}).
}\end{rem}

Note that
\[
1+\beta G_1({\bf x}\.|\.{\bf y}) \ = \ \prod_{j=1}^d\.\bigl(1+\beta(x_j\oplus y_j)\bigr) \ = \
\prod_{i=1}^d(1+\beta x_i) \, \prod_{i=1}^d(1+\beta y_i)\,.
\]
Evaluating both sides at \ts ${\bf x} = {\bf y}_{\lambda}$, we get
\begin{equation} \label{eq:1atalambda}
1+\beta G_1({\bf y}_{\lambda}\.|\.{\bf y}) \ = \ \prod_{i=1}^d \frac{1\. + \.\beta \ts
  y_i}{1\. +\. \beta \ts y_{\lambda_i+d-i+1}}\..
\end{equation}

\bigskip

\section{Hook formula for straight shapes}\label{s:straight}

The goal of this section is to prove the multivariate Theorem~\ref{thm:khlf_multi} and derive
its specializations Theorem~\ref{t:khlf} and Corollary~\ref{cor: q K-HLF}.

\subsection{Multivariate formulas}\label{ss:straight-multi}

First we evaluate ${\bf x}={\bf y}_{\lambda}$ in \eqref{eq:GrothPieri2} and
simplify to obtain the following expression.

\begin{prop}\label{prop:groth_pieri_1}  We have:
\begin{equation} \label{eq:GrothPieri3}
G_{\mu}({\bf y}_{\lambda}\.|\.{\bf y}) \. \bigl(wt(\lambda/\mu)-1\bigr) \ = \ \sum_{\nu
  \mapsto \mu} \. \beta^{|\nu/\mu|} \. G_{\nu}({\bf y_{\lambda}}\.|\.{\bf y})\.,
\end{equation}
where
\[
wt(\lambda/\mu)\, := \ \prod_{i=1}^d \. \frac{1+\beta y_{\mu_i+d-i+1}}{1+\beta
  y_{\lambda_i+d-i+1}}\,.
\]
\end{prop}

\begin{proof}
We evaluate \eqref{eq:GrothPieri2} at ${\bf x}={\bf y}_{\lambda}$ and multiply by
$\beta$. Note that
\[
\frac{\beta G_1({\bf y}_{\lambda}\.|\.{\bf y}) \. - \. \beta\ts G_1({\bf
    y}_{\mu}\.|\.{\bf y})}{1\. + \. \beta \ts
  G_1({\bf y}_{\mu}\.|\.{\bf y})} \ = \ \frac{1+\beta G_1({\bf
    y}_{\lambda}\.|\.{\bf y})}{1\. + \. \beta \ts G_1({\bf y}_{\mu}\.|\.{\bf y})} \. - \. 1\ts .
\]
By \eqref{eq:1atalambda}, this equals \ts $wt(\lambda/\mu)-1$, as desired.
\end{proof}

\smallskip

\begin{thm}[{Multivariate K-HLF}] \label{thm:khlf_multi}
Fix $d\geq 1$. For every \ts $\la\vdash n$ with $\ell(\la) \leq d$ we have:
\begin{equation} \label{eq:khlf-multi}
\aligned
& \sum_{T \in \SIT(\lambda)}
\prod_{k=1}^{m(T)}  \left(\left[\prod_{i=1}^d
  \frac{1\. + \. \beta \ts y_{\nu_i(T_{< k})+d-i+1}}{1\. + \.\beta \ts y_{\lambda_i+d-i+1}}\right] \. - \.1 \right)^{-1} \\
& \quad = \  \frac{1}{\beta^{n}} \,
\prod_{i=1}^{d} \.
  \bigl(1 \. + \. \beta \ts y_{\lambda_i+d-i+1}\bigr)^{\lambda_i}\, \prod_{(i,j) \in \lambda} \. \frac{1}{y_{d+j-\la'_j} \. - \. y_{\la_i+d-i+1} }\,.
\endaligned
\end{equation}
\end{thm}

\begin{proof}
We apply Proposition~\ref{prop:groth_pieri_1} repeatedly, by taking
\. $\mu \gets \nu(T_{\le k-1})$ \. and \. $\nu \gets \nu(T_{\le k})$, and noting
that \ts $\nu \mapsto \mu$ \ts by equation~\eqref{eq:SIT-maps}.  Since this
is a straight shape, we are starting with the empty partition \ts $\emp = \nu(T_{\le 0})$,
until we eventually reach \ts $\nu(T_{\le k})=\lambda$.
Here we use that the vanishing property Proposition~\ref{prop:Gvanprop} ensures that all
shapes are contained in~$\lambda$. We obtain:
\begin{equation*}
\sum_{T\in \SIT(\lambda)} \. \prod_{k=0}^{m(T)-1} \.
\frac{\be^{\ups(T_{\le k+1}) \. - \. \ups(T_{\le k})}}{wt(\la/\nu^{(k)})\ts - \ts 1} \ = \
\frac{G_{\lambda}({\bf y}_{\la} \.|\. {\bf y})}{G_{\emp}({\bf y}_{\la} \.|\. {\bf y})}\,.
\end{equation*}
Since \ts $G_\emp =1$ \ts and
$$G_\la({\bf y}_{\la} \.|\. {\bf y}) \ =  \
\prod_{(i,j)\in\la} \. \frac{y_{d+j-\la'_j} \ts - \ts y_{\la_i+d-i+1}}{1+\be y_{\la_i+d-i+1}}
$$
by Proposition~\ref{prop:Gvanprop}, the desired statement follows.
\end{proof}

\smallskip

\smallskip

\begin{prop} \label{prop:multi-eval}
Fix $d\geq 1$. For every \ts $\la\vdash n$ with $\ell(\la)\leq d$, we have:
\[
 \left. (-1)^{n} \. G_{\lambda}({\bf y}_{\lambda} \.|\. {\bf y}) \right|_{y_i=i} \ = \
 \prod_{i=1}^{d} \. \frac{1}{(1+\beta(\lambda_i+d-i+1))^{\lambda_i}} \, \prod_{(i,j)\in
  \lambda} \. h(i,j)\,.
\]
\end{prop}

\begin{proof}
This follows directly from Proposition~\ref{prop:Gvanprop}, since for \ts $y_i=i$, \ts $i\ge i$,
we have:
$$\bigl(y_{d+j-\lambda'_j} \ts \ominus \ts
y_{\lambda_i+d-i+1}\bigr) \ = \ \frac{j-\lambda_j' - \lambda_i+i-1}{1 + \beta(\lambda_i+d-i+1)} $$
and \. $h(i,j) \ts = \ts \lambda'_j-i \ts + \ts \lambda_i -j +1$.
\end{proof}

\smallskip

\begin{proof}[Proof of Theorem~\ref{t:khlf}]
This follows from Theorem~\ref{thm:khlf_multi} by substituting $y_i\gets i$,
for all \ts $i \ge 1$.  Indeed, notice that
$$y_{d+j-\la_j'}\. - \. y_{\la_i+d-i+1} \, = \, - (\la_i -j + \la_j'-i +1) \, = \, -h(i,j)\ts,
$$
which implies the result.
\end{proof}

\smallskip

\smallskip

\begin{ex}\label{ex:SIT-straight} {\rm
For \ts $\lambda = (2,2)\vdash 4$, the hook lengths are $3,2,2,1$ as in the tableau~$H$ below.
We have:
\[
G_{22}({\bf y}_{22}\.|\. {\bf y}) \.\big|_{y_1=y_2=1} \ = \ \frac{3\cdot 2\cdot 2 \cdot
  1}{(1+3\beta)^2(1+4\beta)^4}\,.
\]
There are three standard increasing tableaux: \ts $\SIT(\la)=\{A,B,C\}$, as shown below:

\smallskip
{\small
$$
H \, = \,\, \ytableaushort{{*(green)3}{*(green)2},{*(green)2}{*(green)1}}  \qquad\qquad\qquad
\qquad
A \, = \,\, \ytableaushort{{*(yellow)1}{*(yellow)2},{*(yellow)3}{*(yellow)4}}\ ,
\qquad
B \, = \,\, \ytableaushort{{*(yellow)1}{*(yellow)3},{*(yellow)2}{*(yellow)4}}\ ,
\qquad
C \, = \,\, \ytableaushort{{*(yellow)1}{*(yellow)2},{*(yellow)2}{*(yellow)3}} \ .
$$
}
\smallskip

\nin
The terms on the RHS of \eqref{eq:khlf-multi} are
$$
u(A) \, = \, u(B) \ = \
         \frac{(1+3\beta)^3(1+4\beta)^2}{6\beta^4(4+10\beta)}\,, \qquad
u(C) \ = \ - \frac{(1+3\beta)^2(1+4\beta)^2}{3\beta^3(4+10\beta)}\,,
$$
and indeed we have
\[
\beta^4 \ts \bigl(u(A)+u(B)+u(C)\bigr) \ = \ \frac{(1+3\beta)^2(1+4\beta)^4}{12}\,.
\]
}\end{ex}

\medskip

\subsection{An infinite version}\label{ss:straight-inf}
Next we give an equivalent expression for Theorem~\ref{t:khlf} in
terms of increasing tableaux instead of standard increasing tableaux.

\smallskip

\begin{thm}[{Infinite Multivariate K-HLF}]\label{thm:khlf_it1} 
Fix $d\geq 1$. For every \ts $\la\vdash n$ with $\ell(\la)\leq d$,  we have:
\begin{equation}\label{eq:hhlf-it1}
\aligned
& \sum_{T \in \IT(\lambda)} \,
\prod_{k=1}^{m} \, \prod_{i=1}^d \.
\frac{1+\beta y_{\lambda_i+d-i+1}}{1+\beta y_{\nu_i(T_{< k})+d-i+1}} \\
& \quad = \ \frac{1}{ (\beta)^{n} } \, \prod_{i=1}^{d} \.
  \bigl(1+\beta(\lambda_i+d-i+1)\bigr)^{\lambda_i} \,  \prod_{(i,j) \in \lambda} \.
\frac{1}{y_{d+j-\la_j'}\ts -\ts y_{\la_i+d-i+1}}\,.
\endaligned
\end{equation}
\end{thm}

\smallskip

In contrast with~\eqref{eq:khlf-multi}, the sum on the LHS of~\eqref{eq:hhlf-it1} is infinite.
This is somewhat further away from the original~\eqref{eq:hlf}, but
closer in spirit to~\eqref{eq:qhlf}.

\smallskip

\begin{proof}
Rewrite Proposition~\ref{prop:groth_pieri_1} as
$$
G_\mu ({\bf y}_{\la}\.|\. {\bf y} ) \ = \ \sum_{\nu \mapsto\mu \ \text{or} \ \nu=\mu} \.
\beta^{|\nu/\mu|} \,
\frac{G_\nu( {\bf y}_{\la} \mid {\bf y} )}{wt(\la/\mu)}\,.
$$
Now, as in the proof of Theorem~\ref{thm:khlf_multi}, iterate this relation until \ts
$\nu(T_{\le m}) = \lambda$, where $m=m(T)$.
This implies the result.
\end{proof}

\smallskip

By analogy with the previous argument for SITs, we obtain the following
infinite version of~\eqref{eq:khlf}:

\smallskip

\begin{cor}[{Infinite  K-HLF}]\label{c:khlf_it}
Fix $d\geq 1$. For every \ts $\la\vdash n$ with $\ell(\la)\leq d$,  we have:
\begin{equation}\label{eq:hklf-inf}
\aligned
& \sum_{T \in \IT(\lambda)} \,
\prod_{k=1}^{m(T)} \, \prod_{i=1}^d \,
\frac{1+\beta (\lambda_i+d-i+1)}{1+\beta (\nu_i(T_{< k})+d-i+1)} \\
& \qquad  = \ \frac{1}{(-\beta)^n} \,
\prod_{i=1}^{d} \. \bigl(1+\beta(\lambda_i+d-i+1)\bigr)^{\lambda_i}
\, \prod_{(i,j) \in \lambda} \. \frac{1}{h(i,j) }\,.
\endaligned
\end{equation}
\end{cor}

The proof follows verbatim the proof above and will be omitted.

\medskip

\subsection{$q$-analogue}\label{ss:straight-q}
Let us now obtain the $q$-analogue of~\eqref{eq:khlf}.

\smallskip

\begin{thm}[{$q$-K-HLF}]\label{thm:khlfq}
Fix $d\geq 1$. For every \ts $\la\vdash n$ with $\ell(\la)\leq d$,  we have:
\begin{equation} \label{eq:khlfq}
\aligned
& \sum_{T \in \SIT(\lambda)}
\prod_{k=1}^{m(T)} \left(\left[\prod_{i=1}^d
  \frac{1+\beta  q^{\nu_i(T_{< k})+d-i+1}}{1+\beta q^{\lambda_i+d-i+1}}\right] -1\right)^{-1}
\\ & \qquad \ \  = \ \frac{q^{m(\la)}}{\beta^n} \.
\prod_{i=1}^{d} \.
  \bigl(1+\beta q^{\lambda_i+d-i+1}\bigr)^{\lambda_i} \, \prod_{(i,j) \in \lambda}
  \.\frac{1}{1-q^{h(i,j)}} \,.
\endaligned
\end{equation}
\end{thm}

\begin{proof}
Substitute \ts $y_i \gets q^i$ \ts for all \ts $i \ge 1$,
in Theorems~\ref{t:khlf} and~\ref{thm:khlf_it1}. Observe that
$$
y_{d+j-\la_j'}\. - \. y_{\la_i+d-i+1} \, = \, q^{d+j-\la_j'} \bigl(1 - q^{h(i,j)}\bigr)\,,
$$
since \ts $h(i,j) = (\la_j'-j) + (\la_i - i) +1$. Following verbatim the argument above,
this implies the result.
\end{proof}

\smallskip

\begin{proof}[Proof of Corollary~\ref{cor: q K-HLF}]
Letting $\beta \to \infty$ in \eqref{eq:khlfq}, for each term on the LHS we have:
$$
\frac{1+\beta  q^{\nu_i(T_{< k})+d-i+1}}{1+\beta q^{\lambda_i+d-i+1}}
\ \to \  q^{\nu_i(T_{< k}) -\la_i} \ = \ q^{-\nu_i(T_{\ge k})}\..
$$
A product of inverses of such terms over all \ts $1\le i \le d$, gives \. $q^{\ups(T_{\ge k})}$.
Factoring out the leading \ts $\beta^{n}$ \ts terms on both sides
and simplifying the formula, we obtain~\eqref{eq:khlf-SIT}.
\end{proof}

\medskip

\subsection{Evaluations of coefficients} \label{ss:straight-coeff}
We can expand the LHS in~\eqref{t:khlf} as a power series in
$\beta$ and compare the coefficients on both sides.
First, as mentioned in the introduction, we recover the original hook-length
formula~\eqref{eq:hlf} by evaluating the constant terms.

\smallskip

\begin{prop}[{$\beta= 0$ \ts in \ts K-HLF}] \label{p:KHLF-HLF}
The  term at $\beta^{-n}$ in equation~\eqref{eq:khlf}  gives~\eqref{eq:hlf}.
\end{prop}

\begin{proof}
Let \ts $\lambda\vdash n$. Extract the constant term in~\eqref{eq:khlf}, after multiplying both sides by $\beta^n$.
In the RHS, we obtain the product of hooks \.
$\prod_{u\in \lambda} 1/h(u)$. In the LHS, since
\[
\frac{1\. +\. \beta \ts p}{1\. +\. \beta \ts t} \ = \ 1 \. + \. \sum_{i= 1}^\infty \. (p-t) \. (-t)^{i-1} \ts  \beta^i,
\]
then the constant term contains only the summands with \ts $m(T)=n$,
each with weight~$1/n!$ \ts By definition, these summands correspond to \ts $T\in \SYT(\la)$.
Thus~\eqref{eq:khlf} at \ts $\beta=0$ \ts gives the HLF in the form
\[
\sum_{T\in \SYT(\lambda)} \.\frac{1}{n!} \ = \ \prod_{u \in \lambda} \. \frac{1}{h(u)}\,,
\]
as desired.
\end{proof}

\medskip

We conclude with a curious corollary relating standard Young tableaux and
barely standard Young tableaux (see~$\S$\ref{ss:Young-inc}).  Here we are using
\ts $p_2(x_1,\ldots,x_d) \. = \. x_1^2 \ts +\ts \ldots \ts + \ts x_d^2$\ts, a symmetric power sum.
Other notation are the \deff{staircase shape} \ts $\de_d = (d-1,\ldots,1,0)$,
and the \deff{harmonic number} \ts $h_n = 1+\frac{1}{2} + \ldots + \frac{1}{n}$\..

\smallskip

\begin{cor}[coefficient of $\beta^{1-n}$ in K-HLF]  
Fix $d\geq 1$. For every \ts $\la\vdash n$ with $\ell(\la)\leq d$,  we have:
\begin{equation}\label{eq:BSYT}
\aligned
& \sum_{ \nu \subsetneq \lambda} \. f^\nu \ts f^{\la/\nu} \, \frac{ p_2(\nu+\delta_d) }{n-|\nu|} \ - \
\sum_{k=1}^{n} \. (n+k-2) \. \bigl|\BSIT_k(\lambda)\bigr|
\\ & \qquad
= \ f^{\lambda} \left((h_n-1) \. p_2(\lambda+\delta_d) \.  + \. \frac{n( n-d(d+1))}{2} \right).
\endaligned
\end{equation}
\end{cor}

\smallskip

The proof is a lengthy but straightforward calculation of evaluating the coefficient of $\beta^1$
on both sides of~\eqref{eq:khlf}  normalized by~$\beta^n$, and will be omitted.
See~$\S$\ref{ss:finrem-BSIT} for the background on~$\BSIT$s.

\bigskip

\section{Generalized excited diagrams}\label{s:excited}

\subsection{Definitions}\label{ss:excited-def}
Given a set $S\subset \lambda$ we say that \ts $(i,j) \in S$ \ts is
\deff{active} if \ts $(i+1,j)$, \ts $(i,j+1)$, and \ts $(i+1,j+1)$ \ts
  are in \ts $\lambda\setminus S$. For an active \ts $u=(i,j) \in S$, define
\ts $a_u(S)$ \ts to be the set obtained by replacing \ts $(i,j)$ \ts by
\ts $(i+1,j+1)$ \ts in~$S$.  Similarly, define \ts $b_u(S)$ \ts
to be the set obtained by adding $(i+1,j+1)$ to $S$.
We call \ts $a_u(S)$ \ts a \deff{type I excited move} and
\ts $b_u(S)$ \ts a \deff{type II excited move}.

Let $\ED(\lambda/\mu)$ be the set of diagrams obtained from $\mu$
after a sequence of type I excited moves on active cells. These are
called \deff{excited diagrams}.  These diagrams are used in
both Naruse hook-length formula~\eqref{eq:Naruse} and its
$q$-analogue~\eqref{eq:qNHLF}.

Let $\EDS(\lambda/\mu)$ be the set of diagrams obtained from $\mu$
after a sequence of both types of excited moves on active cells. These
are called \deff{generalized excited diagrams}.   For example,
the skew shape \ts $\lambda/\mu=43/2$ \ts has five generalized
excited diagrams, three of which are the ordinary excited diagrams.
These are illustrated in Figure~\ref{fig:eds_paths} below.

\medskip

\subsection{Properties}\label{ss:excited-back}
For an excited diagram $D\in \ED(\lambda/\mu)$ we associate a subset \ts
$\pe(D) \subseteq\lambda\setminus D$ \ts called \deff{excited peaks},
constructed inductively, see~\cite[$\S$6.3]{MPP1}. For \ts $\mu\in\ED(\lambda/\mu)$,
let \ts $\pe(\mu)=\emp$. Let \ts
$D\in \ED(\la/\mu)$ \ts be an excited diagram with active cell $u=(i,j)$,
and let \ts $D'= a_u(D)$ \ts be result of the type~I excited
move \ts $D\to D'$. Then the excited peaks of~$D'$ are defined as
\[
\pe(D') \, := \, \pe(D) \. - \. (i,j+1) \. - \. (i+1,j) \. + \. (i,j)\ts,
\]
see Figure~\ref{fig: oed shape 332/21}. It is easy to see that the set~$\pe(D)$ of excited peaks is well defined and independent
on the order of the moves.  Naruse--Okada gave in \cite[Prop.~3.7]{NO} an explicit non-recursive description of~$\pe(D)$ as well as the following characterization of generalized excited diagrams in terms of excited diagrams and excited peaks.

\smallskip

\smallskip

\begin{prop}[{\cite[Prop.~3.13]{NO}}] \label{prop: chara gen ED with EP}
We have:
\[
  \EDS(\lambda/\mu) \ = \ \bigcup_{D \in \ED(\lambda/\mu)} \, \bigl\{ D \cup S~:~S
  \subseteq \pe(D)\bigr\}\ts,
\]
so in particular
\begin{equation}\label{eq:NO-EDS-count}
\big|\EDS(\lambda/\mu)\big| \ = \ \sum_{D \in \ED(\lambda/\mu)} \. 2^{|\pe(D)|}\,.
\end{equation}
\end{prop}

\smallskip


\begin{rem}{\rm There is a certain duality between the set \ts
$\EDS(\la/\mu)$ \ts of generalized excited diagrams and the set \ts
$\mathcal{P}(\lambda/\mu)$ \ts of \emph{pleasant diagrams} defined
in~\cite{MPP1} to give an $\RPP(\la/\mu)$ version
of~\eqref{eq:qNHLF}.  In particular, the following result is
a direct analogue of Proposition~\ref{prop: chara gen ED with EP}. }
\end{rem}

\smallskip

\begin{prop}[{\cite[$\S$6.2]{MPP1}}]
We have:
\[
\mathcal{P}(\lambda/\mu) \ = \ \bigcup_{D \in \ED(\lambda/\mu)} \. \bigl\{
\pe(D) \cup S~:~S \subseteq \lambda \setminus D \bigr\}\ts,
 \]
so in particular
 \begin{equation}\label{eq:MPP-Pleasant-count}
 \bigl|\mathcal{P}(\lambda/\mu)\bigr| \ = \
 \sum_{D\in \ED(\lambda/\mu)} \. 2^{|\lambda/\mu| \. - \. |\pe(D)|} \,.
 \end{equation}
\end{prop}

\smallskip

\begin{ex} \label{ex: eds for two shapes} {\rm
We have \ts \ts $|\ED(332/21)|=5$, see Figure~\ref{fig: oed shape 332/21}, giving
\ts $|\EDS(332/21)| = 11$ by~\eqref{eq:NO-EDS-count}.  Similarly,
equation~\eqref{eq:MPP-Pleasant-count} gives \ts
$|\mathcal{P}(332/21)| = 88$ \ts pleasant diagrams in this case.
}\end{ex}

\begin{figure}[hbt]
    \centering
    \includegraphics[width=8.4cm]{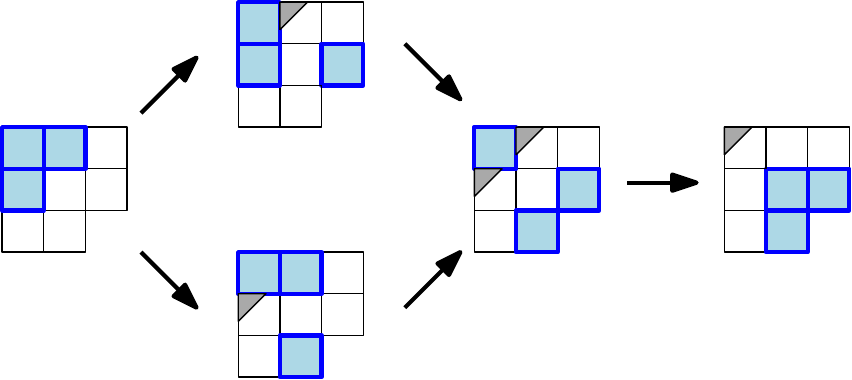}
    \caption{Excited diagrams of shape $\lambda/\mu=332/21$,
    excited moves of type~I, and the corresponding excited
    peaks denoted by shaded triangles.}
    \label{fig: oed shape 332/21}
\end{figure}

\medskip

\medskip

\subsection{Lattice paths interpretation}\label{ss:excited-lattice}
Following the approach in~\cite{K1, MPP2}, these generalized excited diagrams
are in bijection with certain collections of lattice paths by the following
construction.

Let us cut the skew diagram \ts $\la/\mu$ \ts into \emph{border strips}
greedily starting from~$\mu$.   Consider these strips between the  diagonal
starting at $(0,\ell(\mu))$ and the diagonal starting at $(\mu_1,0)$. Within this region, let these border strips have starting squares with midpoints $A_i$ and ending at square with midpoint $B_i$, see Figure~\ref{fig:paths2} (left).

\begin{figure}[hbt]
  \raisebox{15pt}{\includegraphics{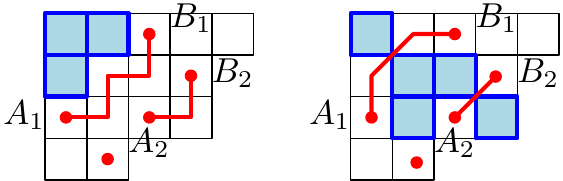}} \qquad  \includegraphics{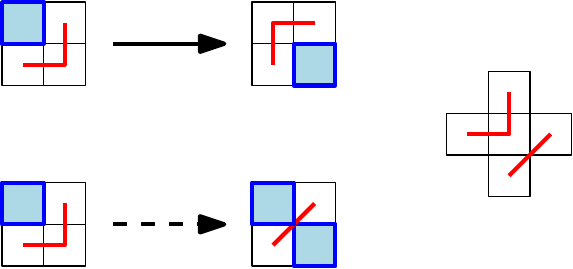}
  \caption{Paths corresponding to two generalized excited diagrams, the flips of the paths in the type I and II excited moves, and  the forbidden path configuration.
  } \label{fig:paths2}
\end{figure}

Let $\eta(A,B)$ be the number of paths \ts $A\to B$ inside $\lambda$,
with endpoints in the center of the squares of the Young diagram and Delannoy steps.
We call these \deff{Delannoy paths}.  The following result interprets the generalized
excited diagrams $\EDS(\lambda/\mu)$ as collections of nonintersecting Delannoy paths
inside $\la/\mu$.

\smallskip

\begin{prop}\label{prop:eds_det}
The set \ts $\EDS(\la/\mu)$ \ts is in bijection with Delannoy path collections \ts $\ga_i: A_i \to B_i$,
such that no two such lattice paths~$\ga_i$ and $\ga_j$  intersect or have configuration as in
Figure~\ref{fig:paths2} $($right$)$.
In particular, we have:
$$
\bigl|\EDS(\la/\mu)\bigr| \ \leq \ \det\bigl[ \eta(A_i,B_j) \bigr]_{i,j}\ .
$$
\end{prop}

\begin{proof}
For the first part, take Delannoy paths in the complement as shown in Figure~\ref{fig:eds_paths}.
Observe that the initial configuration \ts $\mu\in \EDS$, the lowest such lattice paths traverse $\mu$
inside~$\la/\mu$. A type~I excited move transforms a path by flipping a corner from \ts $(1,0),(0,1)$ \ts steps to \ts $(0,1),(1,0)$ \ts steps. A type~II excited move transforms a path by changing a \ts $(1,0),(0,1)$ \ts corner to a $(1,1)$ step, while the cells SE and NW of that step are empty. Further, a type~I excited move applied to cell $u$ with a diagonal step at its SE corner results in flipping this diagonal to steps $(0,1),(1,0)$ and transferring the diagonal step to nearest SE path. A type~II excited move at a cell $u$ with a diagonal step already present results in modifying the nearest SE as above. See Figure~\ref{fig:paths2} (middle).

The final configuration can be drawn by a greedy traverse of the non-excited cells starting from $A_1$ to $B_1$, see Figure~\ref{fig:eds_paths}. Thus the paths pass exactly through the cells outside~$S$, the corresponding moves are reversible on paths as long as there is no intersection and no forbidden configuration.   For the second part, note that \emph{all} non-intersecting Delannoy paths
are enumerated by the determinant using the Lindstr\"om--Gessel--Viennot (LGV) lemma
(see e.g.~\cite[$\S$5.4]{GJ}), giving the desired determinant inequality.
\end{proof}

\begin{figure}[hbt]
\includegraphics[width=12.4cm]{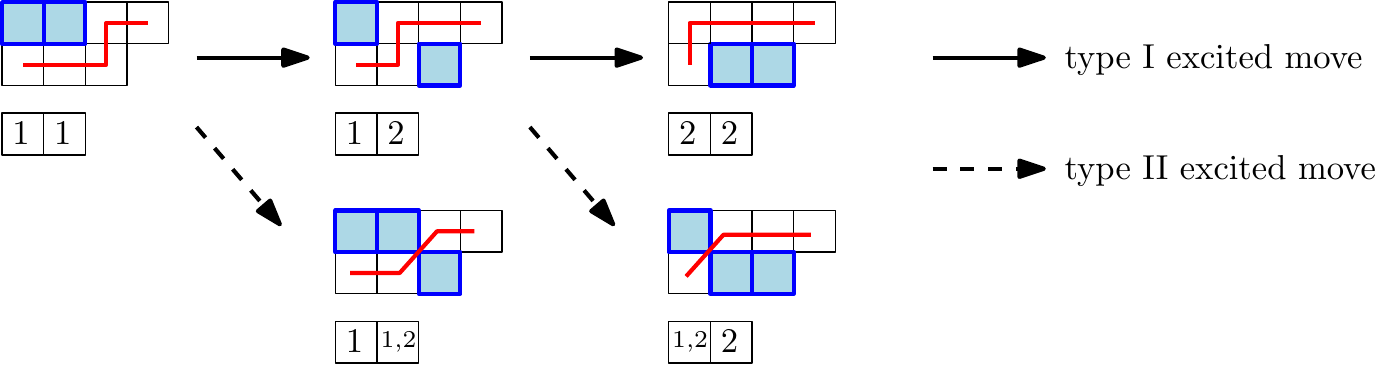}

\caption{The generalized excited diagrams of shape $\lambda/\mu=43/2$,
their peaks and the corresponding flagged set tableaux (see~$\S$\ref{ss:finrem-flag}).
The complements of diagrams in $\EDS(\la/\mu)$ can be viewed as
Delannoy paths inside $\lambda$ (shown in red).}
\label{fig:eds_paths}
\end{figure}

\begin{ex}{\rm
For the skew shape \ts $\lambda/\mu=5442/21$ as in Figure~\ref{fig:paths2}, we have:
\[
23 \. = \. \bigl|\EDS(5442/21)\bigr| \, \leq \, \det \begin{bmatrix}
13 & 7\\
1 & 3
\end{bmatrix} \. = \. 32\ts.
\]
}\end{ex}

\medskip

\subsection{Labeled lattice paths}\label{ss:excited-EDLP}
Kreimain \cite{K1} (see also \cite[Prop.~3.6]{MPP2}), showed that excited diagrams are in bijection with the complements of collections of non-intersecting lattice paths consisting of the $(0,1)$ and $(1,0)$ steps, contained in~$\la$, and with starting and ending points $A_i, B_i$ as above. Note that in \cite{K1,MPP2}, the starting and ending points where different, but the geometry actually forces the corner portions of the paths to be always fixed and hence the start and end points can vary.

Following the definition in~$\S$\ref{ss:Young-paths}, consider the {\em high peaks} of  collection of non-intersecting lattice paths  relative to the original path obtained corresponding to the skew diagram $\la/\mu$. As an example, in Figure~\ref{fig: oed shape 332/21}, there is one lattice path which corresponds to the white cells and the inner corners which are high-peaks are labeled.

\begin{rem}\label{rem:high peaks excited} {\rm
Note that high peaks are a subset of the cells on which a type~I excited move was applied at some point and correspond exactly to the {\em excited peaks}. }
\end{rem}

Denote by \ts $\EDLP(\la/\mu)$ \ts the set of such collections of paths, where each high peak has been labeled $0$ or~$1$.
Similarly, denote by \ts $\DDLP(\la/\mu)$ \ts the set of collections of Delannoy paths in the complement of generalized
excited diagrams in \ts $\EDS(\la/\mu)$.

We can now explain Proposition~\ref{prop: chara gen ED with EP} via lattice
paths by the following bijection \. $\phi: \ts \EDLP(\la/\mu) \to \DDLP(\la/\mu)$ \.
between labeled lattice and Delannoy paths.
Formally, for a collection \ts $\Ups \in \EDLP(\la/\mu)$, replace each high peak
labeled~$1$ with a \ts $(1,1)$ \ts step; all other peaks and paths stay the same.

\smallskip

\begin{prop}\label{prop:eds via labeled ed}
For the a skew shape \ts $\la/\mu$ the map \. $\phi: \ts \EDLP(\la/\mu) \to \DDLP(\la/\mu)$ \. defined above is a bijection.
\end{prop}

\begin{proof}
It is easy to see that for every \ts $\Ups \in \EDLP(\la/\mu)$, the paths in \ts $\phi(\Ups)$ \ts are exactly the Delannoy paths for $\DDLP(\la/\mu)$. For the inverse map~$\phi^{-1}$, replace every \ts $(1,1)$ \ts step with \ts $(0,1),(1,0)$ \ts steps which would necessarily form a high peak and label it~$1$.  This implies the result.
\end{proof}

\medskip

\subsection{Thick zigzag shape} \label{ss:excited-thick-zz}

Consider now the \deff{thick zigzag shape} \ts $\delta_{n+2k}/\delta_n$.  Recall that
\[
\bigl|\ED(\delta_{n+2k}/\delta_n)\bigr| \, = \, \det\bigl[C_{n+i+j-2}\bigr]_{i,j=1}^k
\quad \text{and} \quad \bigl|\mathcal{P}(\delta_{n+2}/\delta_n)\bigr| \. = \.
2^{\binom{k}{2}} \. \det\bigl[\wh s_{n+i+j-2}\bigr]_{i,j=1}^k\,,
\]
where \ts $\wh s_n = 2^{n+2} s_n$. The first equality is proved in \cite[Cor.~8.1]{MPP2},
while the second was originally conjectured in \cite[Conj.~9.3]{MPP1} and proved in \cite[Thm.~1.1]{HKYY}.
We give a similar determinant formula for the number of generalized excited diagrams of
thick zigzag shape.

\smallskip

\begin{thm} \label{thm:EDS thick shapes}
We have: \. $\bigl|\EDS(\delta_{n+2}/\delta_n)\bigr|\ts =\ts s_n$ \. and
\. $\bigl| \EDS(\delta_{n+4}/\delta_{n})\bigr| \. = \. \frac12 \. \bigl(s_n s_{n+2}-s_n^2\bigr)$.
More generally, we have:
\begin{equation}\label{eq:conj-AHM}
\bigl| \EDS(\delta_{n+2k}/\delta_{n})\bigr| \, = \, 2^{-\binom{k}{2}} \. \det \bigl[s_{n-2+i+j} \bigr]_{i,j=1}^k
\quad \text{for all \. $k\geq 1$}\ts.
\end{equation}
\end{thm}

\smallskip

\begin{proof}
 From \cite[\S 3.3, \S 8.1]{MPP2}, the complements of excited diagrams \. $D\in \ED(\delta_{n+2k}/\delta_n)$ \.
 correspond to $k$-tuples $\DD:=(\dd_1,\ldots,\dd_k)$ of non-intersecting Dyck paths \ts
 $\dd_i\in\Dyck(n+2i-2)$, for all \ts $1\le i \le k$, whose set we denote by \ts $\NDyck(n,k)$.
 Define \. $\HP(\DD):=\bigcup_{i=1}^k \HP(\dd_i)$, and \. $\hp(\DD) := |\HP(\DD)|$.

 By Proposition~\ref{prop:eds via labeled ed}, the diagrams \ts $D\in \EDS(\delta_{n+2k}/\delta_n)$ \ts
 correspond to  tuples \ts $(\DD, S)$, where \ts $\DD\in\NDyck(n,k)$ \ts and \ts $S \subseteq \HP(\DD)$ are the high peaks labeled with~$1$.
 We conclude:
\begin{equation} \label{eq: pf EDS thick zigzags}
\bigl|\EDS(\delta_{n+2k}/\delta_n)\bigr| \ = \ \sum_{\DD \in \NDyck(n,k)} 2^{\hp(\DD)}\,.
\end{equation}
Let
\[
L_{n}(x) \,:= \, \sum_{\dd \in \Dyck(n)} x^{\hp(\dd)} \quad\text{and}
\quad L_{n,k}(x)\, := \,\sum_{\DD \in \NDyck(n,k)}  x^{\hp(\DD)}\..
\]
Note that \ts $s_n = L_n(2)$, see e.g.~\cite{Sul}.
By \eqref{eq: pf EDS thick zigzags}, we have \. $L_{n,k}(2) \,=\, \bigl|\EDS(\delta_{n+2k}/\delta_n)\bigl|$.

Finally, by \cite[Thm.~5.9]{HKYY}, the sum \ts $L_{n,k}(x)$ \ts satisfies the following
 determinant formula:
\begin{equation} \label{eq: key det identity thick zigzag}
x^{\binom{k}{2}} \ts \cdot\ts L_{n,k}(x) \ = \  \det\bigl[L_{n+i+j-2}(x)\bigr]_{i,j=1}^k\..
\end{equation}
Setting $x=2$, we obtain the result.
\end{proof}
\medskip

\subsection{Double Grothendieck polynomials}\label{ss:excited-GP}
Excited diagrams can be used to give a combinatorial model of these polynomials
in the special case we need. For a definition and combinatorial models of
\deff{double Grothendieck polynomials} for all permutations,
see~\cite{FK2,FK2b,KM}.

In~\cite{KMY}, Knutson--Miller--Yong gave the following formula for
Grothendieck polynomials of vexillary permutations originally stated in terms of
flagged set tableaux, and restated here in terms of generalized excited diagrams.
See also~$\S$\ref{ss:finrem bumpless} for discussion of another proof of this result.

\smallskip

\begin{thm}[{\cite[Thm.~5.8]{KMY}}] \label{thm: double Groth vexillary}
Let $w$ be a vexillary permutation of shape $\mu$ and supershape~$\lambda$.
Then the double Grothendieck polynomial parameterized by $w$ can be computed as follows:
\begin{equation} \label{eq:GrothKMY}
\mathfrak{G}_{w}({\bf x}\ts,\ts {\bf y})  \ = \ \sum_{D\in \EDS(\lambda/\mu)}
\. \beta^{|D|-|\mu|} \prod_{(i,j)\in D} \. (x_i \ts \oplus \ts y_j)
\end{equation}
\end{thm}

\smallskip

\begin{cor}\label{c:supershape}
Let $w$ be a vexillary permutation of shape $\mu$ and supershape~$\lambda$.
Then we have:

$$
\mathfrak{G}_{w}({\bf x}\ts,\ts{\bf y}) \ = \ \sum_{D\in \ED(\lambda/\mu)} \.
\beta^{|D|-|\mu|} \prod_{(i,j) \in \pe(D)} \bigl(1+ \beta (x_i \ts \oplus \ts y_j) \bigr)
\. \prod_{(i,j)\in D} \bigl(x_i \ts \oplus \ts y_j\bigr).
$$

\end{cor}

\begin{proof}
This follows immediately from Theorem~\ref{thm: double Groth vexillary}
and Proposition~\ref{prop: chara gen ED with EP}.
\end{proof}

\smallskip

\begin{ex}{\rm
For \ts $w=1432\in S_4$, we have \ts $\mu = 21$, \ts $\lambda=332$, and \ts $|\EDS(332/21)|=11$,
see Example~\ref{ex: eds for two shapes} and
\cite[Ex.~1]{FK2}.  By Corollary~\ref{c:supershape} for $y_i =0$  we have:

\begin{align*}
\lefteqn{\mathfrak{G}_{1432}({\bf x}\ts,\ts {\bf 0}) \. = \.}\\
&= \ x_{1}^{2} x_{2}+ x_{2}^{2}x_{1} \left(1+\beta  x_{1}\right)+x_{1}^{2} x_{3} \left(1+ \beta  x_{2}\right)+x_{1} x_{2} x_{3} \left(1+\beta  x_{1}\right) \left(1+\beta  x_{2}\right)+x_{2}^{2} x_{3} \left(1+\beta  x_{1}\right)\\
&= \ {x^{2}_{{1}}}x_{{2}}\, + \,{x^{2}_{{2}}}x_{{1}} \, + \, {x^{2}_{{1}}}x_{{3}}\, + \, x_{{1
}}x_{{2}}x_{{3}} \, + \,{x^{2}_{{2}}}x_{{3}}  \, + \, \beta\ts{x^{2}_{{1}}}{x^{2}_{{2}}
}+2\,\beta\ts{x^{2}_{{1}}}x_{{2}}x_{{3}}\, + \, 2\ts\beta\ts{x^{2}_{{2}}}x_{{1}}x_{{3}}\, + \,{\beta}^{2}{x^{2}_{
{1}}}{x^{2}_{{2}}}x_{{3}}\..
\end{align*}

}
\end{ex}

\medskip

\subsection{Principal specialization}
Let \ts $\Gamma_w(\beta) \ts := \ts \mathfrak{G}_w({\bf 1}\ts,\ts{\bf 0})$ \ts be
the \deff{principal specialization of the Grothendieck polynomial}.
Substituting \ts $x_i  \gets 1$ \ts and \ts $y_i \gets 0$ \ts in Corollary~\ref{c:supershape},
we immediately obtain:

\smallskip

\begin{cor} \label{cor:KMY-groth-excited}
Let $w$ be a vexillary permutation of shape $\mu$ and supershape~$\lambda$.  Then:
\begin{align}
 \Gamma_w(\be)  \ = \, \sum_{D\in \EDS(\lambda/\mu)} \. \beta^{|D|- |\mu|}
 \  = \,\sum_{D \in \ED(\lambda/\mu)}
   \. \beta^{|D|-|\mu|} \, (1+\beta)^{|\pe(D)|} \..
\end{align}
\end{cor}

\medskip

Using the lattice paths interpretation from $\S$\ref{ss:excited-lattice}, let \ts $\eta_{\beta}(A,B)$ be the weighted sum
of Delannoy paths \ts $A\to B$ \ts with $\beta$ keeping track of
the number of \ts $(1,1)$ \ts steps. We have the following inequality for the principal specialization of the Grothendieck polynomials considered above.

\smallskip

\begin{cor} \label{cor:KMY-groth-det}
Let $w$ be a vexillary permutation of shape $\mu$ and supershape~$\lambda$,
and let \ts $\Gamma_w(\beta)$ \ts be
the principal specialization of the Grothendieck polynomial. Then:
$$
 \Gamma_w(\beta)  \ \leqslant \ \det\bigl[\eta_{\beta}(A_i,B_j)\bigr]_{i,j} \  ,
$$
where \. $\leqslant$ \. means coefficient-wise inequality as polynomials in~$\beta$.
\end{cor}

\smallskip

\begin{proof} The result follows immediately from Corollary~\ref{cor:KMY-groth-excited},
the proof of Proposition~\ref{prop:eds_det}, and the proof of the
LGV lemma which preserves the total number of \ts $(1,1)$ \ts steps
under the involution.
\end{proof}

\smallskip
Finally, we give a determinant formula for the principal specialization \ts $\Gamma_{w(n,k)}(1)$,
where
$$w(n,k) \ts := \ts (1,2, \ts \ldots \ts , k,  n+k, n+k-1, \ts \ldots \ts, k+1)\ts.$$
See~\cite{FK1} and \cite[Cor.~5.8]{MPP3}
for the analogous results on evaluations of Schubert polynomials of~$w(n,k)$.

\smallskip

\begin{cor}
For all \ts $n, k\ge 1$, in notation we have:
\[
\Gamma_{w(n,k)}(1) \,=\,
2^{-\binom{k}{2}} \. \det \bigl[s_{n-2+i+j} \bigr]_{i,j=1}^k
\quad \text{for all \. $k\geq 1$}\ts.
\]
\end{cor}

\begin{proof}
The permutation \ts $w(n,k)$ \ts is {\em dominant} ($132$-avoiding),
and hence vexillary. Denote by \ts $\la/\delta_n$ \ts the skew shape associated
to \ts $w(n,k)$, see \cite[Fig.~6(a)]{MPP3}.
Then Corollary~\ref{cor:KMY-groth-excited} at $\beta=1$ gives:
\[
\Gamma_{w(n,k)}(1) \,=\, \bigl| \EDS(\la/\delta_n) \bigr|.
\]
From the definition of generalized excited diagrams, or from their correspondence
with flagged set-valued tableaux (see~$\S$\ref{ss:finrem-flag}),
it is easy to see that \. $\bigl|\EDS(\la/\delta_n)\bigr| = \bigl|\EDS(\delta_{n+2k}/\delta_n)\bigr|$.
The result then follows by Theorem~\ref{thm:EDS thick shapes}.
\end{proof}

\bigskip

\section{Hook formula for skew shapes}\label{s:skew}

\subsection{The setup}
Recall the vanishing property (Proposition~\ref{prop:Gvanprop})
of the factorial Grothendieck polynomials:
\[
G_{\mu}({\bf y}_{\lambda} \.|\. {\bf y}) \ = \ \begin{cases}
0 &\text{ if } \mu\not\subseteq \lambda\.,\\
\prod_{(i,j)\in \lambda} (y_{d+j-\lambda'_j} \ominus
y_{\lambda_i+d-i+1}) &\text { if } \mu = \lambda\..
\end{cases}
\]

\nin
Following the approach of Ikeda--Naruse~\cite{IN1} and Kreiman~\cite{K1} for the
factorial Schur functions \ts $s_{\mu}({\bf y}_{\lambda}\.|\. {\bf y})$, we present
a combinatorial model for the Andersen--Jentzen--Soergel \cite{AJS} and
Billey~\cite{Bil} expressions for evaluations of the
factorial Grothendieck polynomials \ts
$G_{\mu}({\bf y}_{\lambda} \.|\.{\bf y})$ \ts when
$\mu\subseteq \lambda$.

Fix two Grassmannian permutations \ts $w\leq v$ in $S_N$ \ts with
associated partitions $\mu\subseteq \lambda$ with $\ell(\lambda)\leq d$ and $\lambda_1 \leq N-d$,
see e.g.~\cite[$\S$2.1]{Man}.
Let \ts $c_{\mu\tau}^\lambda$ \ts and \ts $K_{\mu\tau}^\lambda$ \ts
be the \emph{structure constants} for the Schubert classes in the
equivariant cohomology and equivariant $K$-theory of the
Grassmannian, respectively, see e.g.~\cite{IN1,K1, GK}.

\smallskip

\begin{thm}[Ikeda--Naruse \cite{IN1}, Kreiman \cite{K1}]  Fix $d\geq 1$. For all \ts $\mu \subset \lambda$ with $\ell(\la)\leq d$, we have:
\[
c_{\mu\lambda}^{\lambda} \ = \ \sum_{D \in \ED(\lambda/\mu)} \, \prod_{(i,j) \in D}
\bigl(y_{d+j-\lambda'_j} \. - \. y_{\lambda_i+d+1-i} \bigr)\ts.\\
\]
\end{thm}

\smallskip

\begin{thm}[{Graham--Kreiman \cite[Thm.~4.5]{GK}}] \label{thm:genexcited2geom}
Fix $d\geq 1$. For all \ts $\mu \subset \lambda$ with $\ell(\la)\leq d$, we have:
\[
K_{\mu\lambda}^\lambda  \ = \  \sum_{D \in \EDS(\lambda/\mu)} (-1)^{|D|-|\mu|} \,
\prod_{(i,j) \in D} \frac{ y_{d+j-\lambda'_j} \. - \. y_{\lambda_i+d+1-i}}{1\. -\. y_{\lambda_i+d+1-i}}\,.
\]
\end{thm}

\smallskip

\begin{rem}{\rm
To translate from the result in \cite[Thm.~4.5]{GK} to the one stated
here one needs to do the  substitution \ts $y_i\gets \bigl(1-e^{\epsilon_i}\bigr)$, as
discussed in \cite[$\S$4.3.1,$\S$5.4]{GK}.
}
\end{rem}

\smallskip

\subsection{Multivariate formulas}  \label{ss:skew-multi}
The following technical lemma gives an evaluation of the factorial Grothendieck polynomials,
and provides a bridge to our enumerative problem.

\smallskip

\begin{lemma} \label{lem:GroExcited}
Fix $d\geq 1$. For all \ts $\mu \subset \lambda$ with $\ell(\la)\leq d$, we have:
\begin{equation}\label{eq:skew-chevalley}
G_{\mu}({\bf y}_{\lambda} \,|\, {\bf y}) \ = \ \sum_{D \in \EDS(\lambda/\mu)} \.
\beta^{|D|-|\mu|} \,
\prod_{(i,j) \in D} \. \bigl(y_{d+j-\lambda'_j} \ts \ominus \ts y_{\lambda_i+d-i+1}\bigr)\ts.
\end{equation}
\end{lemma}

\begin{proof}
We show that both sides of~\eqref{eq:skew-chevalley} satisfy the same identity.
First, the factorial Grothendieck polynomials satisfy the Chevalley
formula \eqref{eq:GrothPieri2}. Thus, for the LHS of~\eqref{eq:skew-chevalley} we have:
\[
G_{\mu}({\bf y}_{\lambda}\.|\.{\bf y})\left(\frac{G_1({\bf
      y}_{\lambda}|{\bf y})\. - \. G_1({\bf y}_{\mu}\.|\.{\bf y})}{1+\beta
    G_1({\bf y}_{\mu}\.|\.{\bf y})} \right) \ = \
\sum_{\nu \supsetneq \mu} \, \beta^{|\nu/\mu|-1} \.
G_{\nu}({\bf y}_{\lambda}\.|\.{\bf y})\..
\]

By Theorem~\ref{thm:genexcited2geom}, the RHS of~\eqref{eq:skew-chevalley} at
\ts $\beta=-1$ \ts equals \ts $K_{\mu\lambda}^{\lambda}$.  On the other hand,
Lenart--Postnikov \cite[Cor.~8.2]{LePo}
(see also the proof of Prop.~3.1 in~\cite{PY}),
give the following \deff{equivariant $K$-theory Chevalley formula}:
\[
K_{\mu\lambda}^{\lambda} \left(\frac{K_{1\ts\lambda}^{\lambda} -1
    + wt'(\mu)}{wt'(\mu)}\right) \ = \ \sum_{\nu \mapsto \mu} \,
(-1)^{|\nu/\mu|-1} \. K_{\nu\lambda}^{\lambda}\,,
\]
where
$$wt'(\mu) \, := \, \prod_{(i,j)\in \mu}
\frac{1-y_{i+j-1}}{1-y_{i+j}}\ts .$$
Observe that we have cancellations in the formula for \ts $wt'(\mu)$, and for each row~$i$
of $\mu$ only the term \ts $(1-y_i)/(1-y_{\mu_i+d-i+1})$ \ts
survives in the product. Thus:
\[
wt'(\mu) \ = \ \prod_{i=1}^d \,
\frac{1-y_i}{1-y_{\mu_i+d-i+1}} \ = \ 1 - G_1\bigl({\bf y}_{\mu}\.|\. {\bf y}\bigr)\mid_{\beta=-1}\,,
\]
where the second equality follows by \eqref{eq:1atalambda}. Therefore, we have:
\[
K_{\mu\lambda}^{\lambda} \left(\frac{K_{1\ts\lambda}^{\lambda} -
    G_1({\bf y}_{\mu}\.|\. {\bf y})\mid_{\beta=-1}}{1 - G_1({\bf
      y}_{\mu}\.|\. {\bf y})\mid_{\beta=-1}}\right) \ = \ \sum_{\nu \supsetneq \mu} \,
(-1)^{|\nu/\mu|-1} \. K_{\nu\lambda}^{\lambda}\,.
\]
This shows that
\[
G_{\mu}({\bf y}_{\lambda} \.|\. {\bf y}) \mid_{\beta=-1} \, = \. K_{\mu\lambda}^{\lambda}.
\]
 We conclude:
\begin{equation} \label{eq:pfeq1}
G_{\mu}({\bf y}_{\lambda} \. | \. {\bf y}) \mid_{\beta=-1} \ \ = \ \sum_{D \in \EDS(\lambda/\mu)} \.
(-1)^{|D|-|\mu|}\, \prod_{(i,j) \in D} \. \frac{ y_{d+j-\lambda'_j} \. - \. y_{\lambda_i+d+1-i}}{1 \. - \. y_{\lambda_i+d+1-i}}\,.
\end{equation}

\smallskip

It remains to show that by substituting \ts $y_i\gets (-y_i\beta)$ \ts in \eqref{eq:pfeq1}, we get the
desired result.
Denote the LHS of \eqref{eq:pfeq1} by $F(y_1,\ldots,y_{n})$. We
easily verify that
\[
(-\beta)^{-|\mu|} \. F(-y_1\beta,\ldots,-y_{n}\beta) \ = \ \sum_{D \in \EDS(\lambda/\mu)} \. \beta^{|D|-|\mu|} \.
\prod_{(i,j) \in D} (y_{d+j-\lambda'_j} \ominus y_{\lambda_i+d-i+1})\,.
\]
Finally, for the RHS by \eqref{eq:betaat-1} we have that
\begin{equation} \label{eq:pfeq2}
G_{\mu}\bigl({\bf y}_{\lambda} \. | \. {\bf y}\bigr)\mid_{y_i\gets (-y_i\beta)} \ \ = \
(-\beta)^{|\mu|}\. G_{\mu}({\bf y}_{\lambda} \. | \. {\bf y})\.,
\end{equation}
as desired.
\end{proof}

\smallskip

\begin{thm}[Multivariate K-NHLF] \label{thm:skewKHLF multi}
Fix $d\geq 1$. For all \ts $\mu \subset \lambda$ with $\ell(\la)\leq d$, we have:
\begin{equation}\label{eq:skewKHLF}
\aligned
& \sum_{T \in \SIT(\lambda/\mu)} \,
\prod_{k=1}^{m(T)} \left(\left[\prod_{i=1}^d \.
  \frac{1 \. +\.\beta \ts y_{\nu_i(T_{< k})+d-i+1}}{1\. + \. \beta \ts y_{\lambda_i+d-i+1}}\right] \. - \. 1\right)^{-1} \\
   & \qquad = \ \sum_{D\in \EDS(\lambda/\mu)} \. \beta^{|D|-|\la|}  \prod_{(i,j)\in
  \lambda \setminus D}  \, \frac{\beta y_{\lambda_i+d-i+1} \. + \. 1}{y_{d+j-\la'_j} \. - \. y_{\la_i+d+1-i}} \,\..
\endaligned
\end{equation}
\end{thm}

\smallskip

\begin{proof}
By Lemma~\ref{lem:GroExcited} and the vanishing property
\eqref{eq:IC} of \. $G_{\mu}({\bf y}_{\mu} \. |\. {\bf y})$, we have:
\begin{equation} \label{eq:pfkhlf1}
\frac{G_{\mu}({\bf y}_{\lambda} \. |\. {\bf y})}{G_{\lambda}({\bf
    y}_{\lambda} \. |\. {\bf y})} = \sum_{D \in \EDS(\lambda/\mu)} \beta^{|D|-|\mu|}
\prod_{(i,j) \in \lambda \setminus D} \frac{1}{y_{d+j-\lambda'_j} \ominus y_{\lambda_i+d-i+1}}.
\end{equation}
Alternatively, by iterating \eqref{eq:GrothPieri3}, we obtain:
\begin{equation} \label{eq:pfkhlf2}
\frac{G_{\mu}({\bf y}_{\lambda}|{\bf y})}{G_{\lambda}({\bf
    y}_{\lambda} \. |\. {\bf y})} \ = \ \beta^{|\lambda/\mu|} \.
    \sum_{T \in \SIT(\lambda/\mu)}
\prod_{k=1}^{m(T)} \left(\left[\prod_{i=1}^d
  \frac{1 \. +\.\beta \ts y_{\nu_i(T_{< k})+d-i+1}}{1\. + \. \beta \ts y_{\lambda_i+d-i+1}}\right] \. - \. 1\right)^{-1}.
\end{equation}
Equating \eqref{eq:pfkhlf1} and \eqref{eq:pfkhlf2} we get the result.
\end{proof}

\smallskip

\begin{proof}[Proof of Theorem~\ref{thm: skew K-NHLF}]
 This follows from Theorem~\ref{thm:skewKHLF multi} by substituting \ts $y_i\gets i$ \ts for all $1\le i \le d$,
 and noticing that \ts $y_{d+j-\la_j'} \ts - \ts y_{\la_i+d-i+1} \. = \.-h(i,j)$.
\end{proof}

\medskip

\subsection{$q$-analogue}  \label{ss:skew-q}
By analogy with the straight shape ($\S$\ref{ss:straight-q}), we obtain a $q$-analogue
using the substitution \ts $y_i\gets q^i$ for all $i\geq 1$.

\smallskip

\begin{thm}[{$q$-K-NHLF}]\label{thm:q-K-NHLF}
Fix $d\geq 1$. For all \ts $\mu \subset \lambda$ with $\ell(\la)\leq d$, we have:
\begin{equation}
\label{eq:skewKHLF}
\aligned
& \sum_{T \in \SIT(\lambda/\mu)}  \,
\prod_{k=1}^{m(T)} \. \left(\left[\prod_{i=1}^d
  \frac{1 \. + \. \beta \ts q^{\nu_i(T_{< k})+d-i+1}}{1 \. + \. \beta \ts q^{\lambda_i+d-i+1}}\right] \. - \. 1\right)^{-1}
  \\
  & \qquad = \  \sum_{D\in \EDS(\lambda/\mu)}  \beta^{|D|-|\la|} \prod_{(i,j)\in
  \lambda \setminus D}  \. \frac{\beta \ts q^{\lambda_i+d-i+1}+1}{q^{d+j-\la'_j} \ts (1 \. - \. q^{h(i,j)})} \,\..
  \endaligned
\end{equation}
\end{thm}

\smallskip

\nin
We omit the proof as the calculations follow verbatim that in the proof of Theorem~\ref{thm:khlfq}.

\smallskip

\begin{proof}[Proof of Theorem~\ref{thm:skewKHLFqbetainfty}]
Following the proof of Corollary~\ref{cor: q K-HLF}, let \ts $\beta \to \infty$ in~\eqref{eq:skewKHLF}.
We have:
$$\frac{1\. + \.\beta\ts q^{\nu_i(T_{< k})+d-i+1}}{1\. + \.\beta\ts q^{\lambda_i+d-i+1}}
\ \to \ q^{-|\la_i - \nu_i(T_{< k})|} \, = \, q^{-|\nu_i(T_{\ge k})|}\,.
$$
Taking the inverse of a product of these terms over all \ts $1\le i\le d$,
we get \ts $q^{\ups(T)}$. The $\beta$ terms on the RHS of~\eqref{eq:skewKHLF}
all have exponents zero, which implies the result.
\end{proof}

\smallskip

Finally, as discussed in the introduction (see Remark~\ref{r:intro-skew-SIT}),
we can now rewrite the RHS of \eqref{eq:skewKHLF} in terms of the (ordinary)
excited diagrams.

\smallskip

\begin{cor} \label{cor: q-skewKHLF in terms of excited}
For every \ts $\mu \ssu \la$,  we have:
\begin{equation} \label{eq:skewKHLF in terms of excited}
\aligned
& \sum_{T \in \SIT(\lambda/\mu)} \. q^{|T|} \,
\prod_{k=1}^{m(T)} \frac{1}{1-q^{\ups(T_{\geq k})}}  \\
& \qquad \ = \  \sum_{D\in
  \ED(\lambda/\mu)} \ \prod_{(i,j) \in \pe(D)} \.\frac{1}{1-q^{h(i,j)}} \. \prod_{(i,j)\in
  \lambda \setminus D} \. \frac{ q^{h(i,j)}}{1 - q^{h(i,j)}} \,.
\endaligned
\end{equation}
\end{cor}


\begin{proof}
This follows from Theorem~\ref{thm:skewKHLFqbetainfty} and the characterization of
generalized excited diagrams given in Proposition~\ref{prop: chara gen ED with EP}.
\end{proof}

\medskip

\subsection{Back to set-valued tableaux}\label{ss:skew-OOF}
The following \deff{Okounkov--Olshanski formula} (OOF) given in~\cite{OO},
is yet another nonnegative formula for \ts $f^{\lambda/\mu}$\ts.  Fix $d \geq 1$ for $\mu\subset \la$ with $\ell(\la)\leq d$ we have: 
\begin{equation} \tag{OOF}\label{eq:oof}
f^{\lambda/\mu} \ =\ n! \sum_{T \in \SSYT_d(\mu)}  \, \prod_{(i,j)\in \lambda}  \bigl(\lambda_{d+1-T(i,j)} \. + \ts i-j\bigr)
 \. \prod_{(i,j) \in \lambda}  \.  \frac{1}{h(i,j)}  \.,
\end{equation}
where \ts $\SSYT_d(\mu)$ \ts denotes the set of SSYTs of shape $\mu$ with entries \ts $\leq d$.
Note that~\eqref{eq:oof} is also proved via evaluations of factorial Schur functions,
preceding~\eqref{eq:Naruse} in this approach.  The corresponding $q$-analogues
are given in~\cite[Thm.~1.2]{CS}  and~\cite[$\S$1.4]{MZ}, for the summations over
\ts $\SSYT(\la/\mu)$ \ts and \ts $\RPP(\la/\mu)$, respectively.

Here we follow a simple proof in~\cite[$\S$3.1]{MZ} via evaluations of factorial Schur functions,
to give a \ts (K-OOF) \ts generalization of~\eqref{eq:oof} for \ts $\SIT(\lambda/\mu)$ \ts analogous to
Theorem~\ref{thm: skew K-NHLF}.

\smallskip

\begin{thm}[{K-OOF}] \label{thm:skew K-OOF}
Fix $d\geq 1$. For all \ts $\mu \subset \lambda$ with $\ell(\la)\leq d$,  we have:
$$
\aligned
&
\sum_{T \in \SIT(\lambda/\mu)}
\prod_{k=1}^{m(T)} \left(\left[\prod_{i=1}^d
  \frac{1+\beta\bigl(\nu_i(T_{< k})+d-i+1\bigr)}{1\. + \. \beta \bigl(\lambda_i+d-i+1\bigr)}\right] -1\right)^{-1}  \ = \ \,
  \prod_{i=1}^d \. \bigl(1\. + \. \beta\ts (\lambda_i +d-i+1)\bigr)^{\lambda_i}
  \\ & \hskip2.5cm \times
\sum_{T\in \SSVT_d(\mu)}  \. (-\beta)^{\nent(T)-|\lambda|} \. \prod_{(i,j) \in \mu, \. r \in T(i,j)} \.
\frac{\lambda_{d+1-r}\. + \ts i \ts - \ts j}{1\. + \. \beta\bigl(\lambda_{d+1-r} \. + \ts r\bigr)} \.
 \prod_{(i,j) \in \lambda} \. \frac{1}{h(i,j)}
 \endaligned
$$
\end{thm}

\begin{proof}
We evaluate \.
$G_{\mu}({\bf y}_{\lambda} \. |\. {\bf y}) \,/\, G_{\lambda}({\bf y}_{\lambda} \. |\. {\bf y}) \mid_{y_i \gets i}$ \.
in two different ways.
First, the LHS is obtained by substitution \ts $y_i \gets i$ \ts in~\eqref{eq:pfkhlf2}.
For the RHS we evaluate the numerator and denominator directly. For the denominator
we use Proposition~\ref{prop:multi-eval}. For the numerator, since \ts
$G_{\mu}(x_1,\ldots,x_d\.|\.{\bf y})$ \ts is symmetric in $x_1,\ldots,x_d$ \ts
by Proposition~\ref{prop:props of fG}~$(i)$,  we have:
 \[
 \aligned
 & G_{\mu}\bigl( \ts \ominus \ts(\lambda_1 +d), \. \ldots \. ,  \. \ominus \ts(\lambda_{d-1}+2), \. \ominus\ts (\lambda_d+1) \, \big| \,\. 1,2,3, \ldots\bigr) \\
 & \qquad =\  G_{\mu}\bigl( \ts \ominus \ts (\lambda_d +1), \. \ominus \ts (\lambda_{d-1}+2), \. \ldots \., \. \ominus \ts (\lambda_1+d) \, \big| \, \.  1,2,3,\ldots\bigr)\ts.
 \endaligned
 \]
 Next, by Definition~\ref{defthm: fG} of factorial Grothendieck polynomials, the RHS of the equation above is equal to
 \[
\sum_{T \in \SSVT_d(\mu)} \. \beta^{\nent(T)-|\mu|} \.
\prod_{(i,j) \in \mu, \. r\in T(i,j)} \. \left[ \frac{-(\lambda_{d+1-r}\. + \. r)}{1+\beta(\lambda_{d+1-r} \. + \. r)} \ts \oplus \ts (r+j-i)\right].
 \]
 The result then follows by simplifying power of $\beta$ and doing the calculation
 \[
 \frac{-(\lambda_{d+1-r} \. + \. r)}{1\. + \. \beta\ts (\lambda_{d+1-r} \. + \. r)} \. \oplus \. (r+j-i)
 \ = \   \frac{-\lambda_{d+1-r} \. - \. i \ts + \ts j }{1\. + \. \beta\bigl(\lambda_{d+1-r}\. + \. r\bigr)}\,.
 \]
 We omit the details.
\end{proof}

\smallskip

\begin{rem}{\rm
Note that the set \ts $\SSYT_d(\mu)$ \ts in~\eqref{eq:oof} is finite and plays
a role of the set \ts $\ED(\la/\mu)$ \ts of excited diagrams in~\eqref{eq:Naruse}.
This connection is clarified in~\cite{MZ}, with reformulations of~\eqref{eq:oof}
in terms of {\em puzzles} and {\em reverse excited diagrams}.  Finally,
the set \ts $\SSVT_d(\mu)$ \ts plays a role of generalized excited diagrams
\ts $\EDS(\la/\mu)$.  It would be interesting to reformulate the theorem
similarly, in terms of puzzles.
}\end{rem}
\bigskip

\section{Final remarks and open problems} \label{s:finrem}

\subsection{} \label{ss:finrem-hist}  The hook-length formula~\eqref{eq:hlf}
has numerous proofs, starting with the original paper~\cite{FRT}.  The
Littlewood formula~\eqref{eq:qhlf} was first given in \cite[p.~124]{Lit}.
We refer to \cite[$\S$6.2]{CKP} for an overview of other proofs and generalizations.
The Naruse hook-length formula~\eqref{eq:Naruse} was originally given by Naruse in his
talk slides~\cite{Strobl}. In our first two papers of this series
\cite{MPP1,MPP2} we give about four proofs of this result, which include
both the \ts $\SSYT$ \ts and \ts $\RPP$ \ts generalizations,
see~\eqref{eq:qNHLF} and~\eqref{eq:RPP-skew}.

\subsection{} \label{ss:finrem-asy}
In~\cite{MPP3}, we give various enumerative and asymptotics applications
of the~\eqref{eq:Naruse}.  Further applications and comparisons with
other tools for estimating \ts $f^{\la/\mu} = |\SYT(\la/\mu)|$ \ts are surveyed
in~\cite{Pak2}.  It would be interesting to find
similar applications of the \ts $\be$-deformations presented
in this paper. Let us single out Thm.~3.10 in~\cite{MPP3} which established
a key symmetry via factorial Schur functions, used to obtain a host of
product formulas.  Note that two elementary proofs of this result are
given in~\cite{PP}; we are especially curious to find its generalization
motivated by the factorial Grothendieck polynomials.

\subsection{} \label{ss:finrem-factorial}
The notation used for the factorial Grothendieck polynomials
goes back to the formal group law of \emph{connective $K$-theory},
and in the context of  Algebraic Combinatorics
is explained in~\cite{FK2} as follows.

Let \ts $\mathcal{A}_n^{\beta}$ \ts be
the algebra with generators \ts $u_1,\ldots,u_{n-1}$ \ts satisfying
\ts $u_i^2= \beta u_i$, the exchange and braid relation.
Observe that \ts $\mathcal{A}_n^{0}$ \ts is the \emph{NilCoxeter algebra}
and \ts $\mathcal{A}_n^{-1}$ \ts is the \emph{degenerate Hecke algebra}.
Then the functions \ts $h_i(t)=e^{tu_i}$ \ts satisfy the \emph{Yang--Baxter
equation}:
\[
h_i(t)\ts h_{i+1}(t+s) \ts h_i(s) \, = \, h_{i+1}(s) \ts h_i(t+s) \ts h_{i+1}(t)\ts.
\]
For \ts $h_i(t)=e^{tu_i}=1+xu_i$ \ts we have \ts $x=(e^{\beta t}-1)/\beta$.
We can now write this as \ts $x=[t]_{\beta}$ \ts and note that \ts
$[t]_{\beta} \oplus [s]_{\beta} = [t+s]_{\beta}$.

\subsection{} \label{ss:finrem-BSIT}
Our notion of \emph{barely standard Young tableaux} \ts $\BSIT$ \ts comes
from a similar notion of \emph{barely set-valued tableaux} recently introduced
in~\cite{RTY}, and probably the closest relative of~$\SYT$ that we have.
Note that~\eqref{eq:BSYT} can be rewritten as computing the
expectation of the repeated entry, similar to~\cite{RTY} (see also~\cite{FGS}),
although the resulting formula is more cumbersome.

\subsection{} \label{ss:finrem-flag}
Excited diagrams are in bijection with certain {\em flagged tableaux}: \.
$|\ED(\la/\mu)| = |\Flag(\la/\mu)|$, where \.
$\Flag(\la/\mu)\ssu \SSYT(\mu)$, see \cite[$\S$3.3]{MPP1}.
This connection was used in~\cite[$\S$3.3]{MPP2} to obtain a determinant formula
for~$|\ED(\la/\mu)|$.
Similarly, the generalized excited diagrams in \ts $\EDS(\lambda/\mu)$ \ts
are in bijection with certain \emph{flagged set-valued tableaux} of shape~$\mu$,
see an example in Figure~\ref{fig:eds_paths}. These bijections were obtained
by Kreiman~\cite[$\S$6]{K1} and by Knutson--Miller--Yong \cite[$\S$5]{KMY}
in the context of Schubert calculus.

\subsection{}\label{ss:finrem-detEDS}
In Theorem~\ref{thm:EDS thick shapes}, we gave a determinant formula for the number of generalized excited diagrams of the skew shape \ts $\delta_{n+2k}/\delta_n$ \ts using the connection between \ts $\EDS(\lambda/\mu)$ \ts and \ts $\mathcal{P}(\lambda/\mu)$, see Proposition~\ref{prop: chara gen ED with EP}.  A similar determinant formula for \ts $\mathcal{P}(\delta_{n+2k}/\delta_n)$ \ts is proved in~\cite{HKYY}. In fact, \cite[Cor.~6.4]{HKYY} gives determinant formulas for pleasant diagrams of more general classes of skew shapes called {\em good} that also include {\em thick reverse hooks} \ts $(b+c)^{a+c}/b^a$. Using \cite[Thm.~6.3]{HKYY}, which is an analogue of~\eqref{eq: key det identity thick zigzag}, one can show determinant formula for generalized excited diagrams of such good skew shapes.

\subsection{}\label{ss:finrem bumpless}
In~\cite[Cor.~1.5, Thm.~1.1]{W}, Weigandt gave two formulas for double Grothendieck polynomials \ts $\mathfrak{G}_w({\bf x}\ts,\ts {\bf y})$ \ts in terms of the {\em bumpless pipe dreams} of $w$ defined by Lam--Lee--Shimozono \cite{LLS}.  When $w$ is vexillary, these formulas reduce to Theorem~\ref{thm: double Groth vexillary}  and Corollary~\ref{c:supershape}, respectively.  Indeed,
a bijection between {\em marked} bumpless pipe dreams of vexillary~$w$ and \ts $\EDS\bigl(\lambda(w)/\mu(w)\bigr)$ \ts via the corresponding flagged set-valued tableaux is given in~\cite[Thm.~1.6]{W}.  Similarly, a bijection between vexillary bumpless pipe dreams
and ordinary excited diagrams is given in~\cite[$\S$7.3]{W}.

We should mention that bumpless pipe dreams of $w$ behave like (generalized) excited diagrams of shape $\lambda/\mu$, since the former are connected by certain moves called ($K$-theoretic) {\em droop moves} \cite{LLS,W}. It would be interesting to further explore this connection.

\subsection{}\label{ss:finrem-IT}
There is a large literature on enumeration of increasing tableaux in
many special cases based on a trick of adding \ts $M_\la$ \ts implicitly
used in~\eqref{eq:intro-IT}.  Notably, for the rectangular shape, tableaux in \ts $\SIT(a^b)$ \ts
 are in bijection with certain plane partitions of the same shape,
see e.g.~\cite[$\S$4]{DPS} and~\cite{HPPW}. This approach fails to give
a bijection for general skew shapes $\la/\mu$, except when $\mu=\de_k$
is a staircase.  The latter are characterized by all minimal elements in \ts
$M_{\la/\mu}$ \ts having the same entries.

\subsection{}\label{ss:finrem-q-inf}
While all our proofs are algebraic, some of our results seem well-positioned
to have a direct combinatorial proof.  We are especially curious
if~\eqref{eq:khlf} has such a proof.
Similarly, it would be interesting to use Konvalinka's recursive
approach~\cite{Kon}, to find a combinatorial proof of
our Theorem~\ref{thm: skew K-NHLF}.

\subsection{}\label{ss:finrem-CS}
The complexity of counting standard increasing tableaux is yet to
be understood.  In~\cite[$\S$1.3]{TY3}, the authors give examples of
large primes appearing as values, and suggest that the exact formula
might not exist.  They ask if there are ``efficient (possibly randomized
or approximate) counting algorithms'' for \ts $g^\la=|\SIT(\la)|$ \ts and its
refinements.

We conjecture that computing \ts $g^\la$ \ts
is $\SP$-complete.  This would partly explain why our hook formulas
involve nontrivial $\beta$-weights. For the related notion of
set-valued tableaux, see a discussion in~\cite{MPY} and $\SP$-completeness
conjecture in~\cite[$\S$5.7]{H+}.

\subsection{}\label{ss:finrem-det-AF}
The LHS of~\eqref{eq:khlf} is equal to the LHS of equation~(K-OOF)
given in Theorem~\ref{thm:skew K-OOF}. It then follows from the proof
of Theorem~\ref{thm:skew K-OOF} that both can be computed efficiently
for a given skew shape \ts $\la/\mu$ \ts and \ts $\be\in \qqq$.
It would be interesting to see if these
have a determinant formula generalizing the \emph{Aitken--Feit determinant
formula} for \ts $f^{\la/\mu}$ \ts (see e.g.~\cite[Cor.~7.16.3]{St2} and~\cite{Pak2}).

Note that the \emph{Lascoux--Pragacz identity} gives yet another determinant
formula for \ts $f^{\la/\mu}$, which we used in~\cite{MPP2} to give a
combinatorial proof of~\eqref{eq:Naruse}.  Finally, let us mention that
\ts $\ED(\la/\mu)$ \ts has a determinant formula (see~$\S$\ref{ss:finrem-flag}i above),
while Proposition~\ref{prop:eds_det} is not an equality but gives only a
determinant upper bound for \ts $\EDS(\la/\mu)$.

\subsection{}\label{ss:finrem-principal}
Following the approach of Stanley~\cite{St-shenanigans}, we conjecture
that for all \ts $\beta\ge 0$, there is a limit
$$\lim_{n\to \infty} \, \frac{\log_2 u(\beta,n)}{n^2}\,, \qquad \text{where} \qquad
u(\beta,n) \, := \, \max_{w\in S_n} \. \Ga_w(\beta)\ts.
$$
Using the \emph{Cauchy identity for Grothendieck polynomials}~\cite[Cor.~5.4]{FK2},
we obtain the following bounds:
\[
\frac{1}{4} \ts \log_2 (2+\beta) \, \leq \, \liminf_{n\to \infty} \, \frac{\log_2 u(\beta,n)}{n^2}  \, \leq \,
\limsup_{n\to \infty} \. \frac{\log_2 u(\beta,n)}{n^2} \, \leq \, \frac{1}{2} \ts \log_2 (2+\beta).
\]
In~\cite{MPP5}, we computed the limit above for \ts $\be=0$, when the maximum is restricted to \emph{layered}
($231$- \ts and \ts $312$-avoiding) \emph{permutations}.  It would be interesting to see if our
analysis can be extended to the case of general \ts $\be>0$.

\subsection{}\label{ss:finrem-sampling}
Dividing both sides of~\eqref{eq:khlf} by \ts $(-1)^n$ \ts and taking \ts $\be>0$,
gives positive weights in the summation on the LHS over the~$\SIT$s. Can one efficiently
sample from this distribution?  Perhaps, there is a deformation of the
\emph{NPS algorithm} or the \emph{GNW hook walk}?  A positive answer
to either of these would be remarkable.
\vskip.7cm

{\small

\subsection*{Acknowledgements}
This paper is dedicated to the memory of Sergei Kerov, whose work
was inspirational to all of us.
While the results in this paper were obtained over five years ago,
the writing was greatly delayed due to various life related matters.
It took the $\Delta$ variant to hold off other plans and finally force
us to finish this project.

We are thankful to Zach Hamaker, Oliver Pechenik, Pasha Pylyavskyy,
Anna Weigandt, Damir Yeliussizov and Alex Yong
for interesting conversations and helpful remarks on the paper.
AHM was partially supported by the NSF grant DMS-1855536.
IP was partially supported by the NSF grants DMS-1700444 and CCF-2007891.
GP was partially supported by the NSF grant DMS-1939717 and CCF-2007652.
}



\newpage

\begin{thebibliography}{1000011}

\bibitem[AJS]{AJS}
H.~H.~Andersen, J.~C.~Jantzen and W.~Soergel,
Representations of quantum groups at $p$-th root of unity and of
semisimple groups in characteristic~$p$: independence of~$p$,
\emph{Ast\'{e}risque}~\textbf{220} (1994), 321~pp.

\bibitem[Bil]{Bil}
S.~Billey,
{K}ostant polynomials and the cohomology ring for~$G/B$,
{\em Duke Math.~J.} \textbf{96} (1999), 205--224.


\bibitem[Bri]{Bri}
M.~Brion,
Lectures on the geometry of flag varieties, in
\emph{Topics in cohomological studies of algebraic varieties},
Birkh\"auser, Basel, 2005, 33--85.

\bibitem[B1]{Bu}
A.~Buch,
A {L}ittlewood--{R}ichardson rule for the {$K$}-theory of
{G}rassmannians, {\em Acta Math.}~{\bf 189} (2002), 37--78.

\bibitem[B2]{Buch}
A.~Buch, Combinatorial $K$-theory, in
\emph{Topics in cohomological studies of algebraic varieties},
Birkh\"auser, Basel, 2005, 87--103.

\bibitem[B+]{BuKSTY}
A.~Buch, A.~Kresch, M.~Shimozono, H.~Tamvakis and A.~Yong,
Stable {G}rothendieck polynomials and $K$-theoretic factor sequences,
{\em Math.\ Ann.}~{\bf 340}  (2008), 359--382.

\bibitem[BMN]{BMN}
D.~Bump, P.~J.~McNamara and M.~Nakasuji,
Factorial Schur functions and the Yang--Baxter equation,
\emph{Comment.\ Math.\ Univ.\ St.~Pauli}~\textbf{63} (2014), 23--45.

\bibitem[CS]{CS}
X.~Chen and R.~P.~Stanley,
A formula for the specialization of skew Schur functions,
\emph{Ann.\ Comb.}~\textbf{20} (2016), 539--548.

\bibitem[CKP]{CKP}
I.~Ciocan-Fontanine, M.~Konvalinka and I.~Pak,
The weighted hook length formula,
\emph{J.~Combin.\ Theory, Ser.~A}~\textbf{118} (2011), 1703--1717.

\bibitem[DPS]{DPS}
K.~Dilks, O.~Pechenik and J.~Striker,
Resonance in orbits of plane partitions and increasing tableaux,
\emph{J.\ Combin.\ Theory, Ser.~A}~\textbf{148} (2017), 244--274.

\bibitem[EG]{EG}
P.~Edelman and C.~Greene, Balanced tableaux,
{\em Adv.\ Math.}~{\bf 63} (1987), 42--99.

\bibitem[FGS]{FGS}
N.~J.~Y.~Fan, P.~L.~Guo and S.~C.~C.~Sun,
Proof of a conjecture of Reiner--Tenner--Yong on barely set-valued tableaux,
\emph{SIAM J.\ Discrete Math.}~\textbf{33} (2019), 189--196.

\bibitem[FK1]{FK2}
S.~Fomin and A.~N.~Kirillov,
Yang--Baxter equation, symmetric functions and Grothendieck polynomials,
preprint (1993), 25~pp.; \ts {\tt arXiv:hep-th/9306005}.
%

\bibitem[FK2]{FK2b}
S.~Fomin and A.~N.~Kirillov,
Grothendieck polynomials and the Yang--Baxter equation,
in \emph{Proc.\ 6-th FPSAC}, DIMACS, Piscataway, NJ, 1994, 183--190.

\bibitem[FK3]{FK1}
S.~Fomin and A.~N.~Kirillov,
Reduced words and plane partitions,
{\em J.\ Algebraic Combin.}~{\bf 6} (1997), 311--319.

\bibitem[FRT]{FRT}
J.~S.~Frame, G.~de~B.~Robinson and R.~M.~Thrall,
The hook graphs of the symmetric group,
{\em Canad.\ J.\ Math.}~\textbf{6} (1954), 316--324.


\bibitem[GJ]{GJ}
I.~P.~Goulden and D.~M.~Jackson, \emph{Combinatorial enumeration},
Wiley, New York, 1983, 569~pp.

\bibitem[GK]{GK}
W.~Graham and V.~Kreiman,
Excited {Y}oung diagrams and equivariant $K$-theory,
and Schubert varieties,
\newblock {\em Trans.\ AMS}~{\bf 367} (2015), 6597--6645.

\bibitem[HPPW]{HPPW}
Z.~Hamaker, R.~Patrias, O.~Pechenik and N.~Williams,
Doppelg\"angers: bijections of plane partitions,
\emph{IMRN} (2020), no.~2, 487--540.

\bibitem[H+]{H+}
Z.~Hamaker, A.~H.~Morales, I.~Pak, L.~Serrano and N.~Williams,
Bijecting hidden symmetries for skew staircase shapes,
preprint (2021), 19~pp.; \ts {\tt arXiv:2103.09551}.

\bibitem[HY]{HY}
Z.~Hamaker and B.~Young,
Relating Edelman--Greene insertion to the Little map,
 {\em J.\ Algebraic Combin.}~{\bf 40}  (2014), 693--710.

\bibitem[HKYY]{HKYY} B.~H.~Hwang, J.~S.~Kim, M.~Yoo and S.~M.~Yun,
Reverse plane partitions of skew staircase shapes and $q$-{E}uler numbers,
{\em J.~Combin.\ Theory, Ser.~A}~\textbf{168} (2019), 120--163.

\bibitem[IN1]{IN1}
 T.~Ikeda and H.~Naruse,
 Excited {Y}oung diagrams and equivariant {S}chubert calculus,
 {\em Trans.\ AMS}~\textbf{361} (2009), 5193--5221.

\bibitem[IN2]{IN2}
T.~Ikeda and H.~Naruse,
$K$-theoretic analogues of factorial Schur $P$- and $Q$-functions,
 {\em Adv.\ Math.}~\textbf{243}  (2013), 22--66.

\bibitem[Kre]{K1}
V.~Kreiman,
{S}chubert classes in the equivariant $K$-theory and equivariant
  cohomology of the {G}rassmannian, preprint (2005), 27~pp.; {\tt arXiv:math.AG/0512204}.

\bibitem[KM]{KM}
A.~Knutson and E.~Miller,
Gr\"obner geometry of Schubert polynomials,
{\em Annals of Math.}~\textbf{161} (2005), 1245--1318.

\bibitem[KMY]{KMY}
 A.~Knutson, E.~Miller and A.~Yong,
 Gr\"obner geometry of vertex decompositions and of flagged tableaux,
 {\em J.~Reine Angew.\ Math.}~\textbf{630} (2009), 1--31.


\bibitem[Kon]{Kon}
M.~Konvalinka,
A bijective proof of the hook-length formula for skew shapes,
\emph{European J.\ Combin.}~\textbf{88} (2020), 103104, 14~pp.

\bibitem[LLS]{LLS}
T.~Lam, S.~J.~Lee, and M.~Shimozono,
Back stable Schubert calculus,
\emph{Compositio Math.}~\textbf{157} (2021), 883--962.

\bibitem[LS1]{LS1}
A.~Lascoux and M.-P.~Sch\"utzenberger,
Polyn\^{o}mes de {S}chubert (in French),
\emph{C.~R.\ Acad.\ Sci.\ Paris} \textbf{294} (1982), no.~13, 447--450.

\bibitem[LS2]{LS2}
A.~Lascoux and M.-P.~Sch\"utzenberger,
Structure de Hopf de l'anneau de cohomologie et de l'anneau de Grothendieck d'une
vari\'et\'e de drapeaux (in French),
\emph{C.~R.\ Acad.\ Sci.\ Paris}~\textbf{295} (1982), no.~11, 629--633.

\bibitem[LP]{LePo}
C.~Lenart and A.~Postnikov, Affine {W}eyl groups in $K$-theory and
representation theory, {\em IMRN} (2007), no.~12, Art.\ ID~rnm038, 65~pp.

\bibitem[Lit]{Lit}
D.~E.~Littlewood,
\emph{The theory of group characters and matrix representations of groups}
(second ed.), Oxford Univ.\ Press, New York, 1950, 310~pp.

\bibitem[Mac]{Mac}
I.~G.~Macdonald,
Schur functions: theme and variations, in \emph{Publ.\ IRMA}~\textbf{498},
Univ.\ Louis Pasteur, Strasbourg, 1992, 5--39.

\bibitem[Man]{Man}
L.~Manivel,
{\em Symmetric functions, Schubert polynomials and degeneracy loci},
AMS, Providence, RI, 2001, 167~pp.

\bibitem[Mc1]{McN}
P.~J.~McNamara,
Factorial Grothendieck polynomials,
{\em Electron.\ J.\ Combin.}~{\bf 13} (2006), no.~1, RP~71, 40~pp.

\bibitem[Mc2]{McNadd}
P.~J.~McNamara, Addendum to Factorial Grothendieck polynomials, preprint (2011), 2~pp.;
available at \ts \href{http://petermc.net/maths/papers/fgpaddendum.pdf}{tinyurl.com/3dfs7e9s}\ts.

\bibitem[MS]{MS}
A.~I.~Molev and B.~E. Sagan,
A {L}ittlewood--{R}ichardson rule for factorial {S}chur functions,
{\em Trans.\ AMS}~\textbf{351} (1999), 4429--4443.

\bibitem[MPY]{MPY}
C.~Monical, B.~Pankow and A.~Yong,
Reduced word enumeration, complexity, and randomization, preprint (2019), 23~pp.;
\ts {\tt arXiv:1901.03247}.

\bibitem[MPP1]{MPP1}
A.~H.~Morales, I.~Pak and G.~Panova,
Hook formulas for skew shapes I. $q$-analogues and bijections,
\emph{J.~Combin.\ Theory, Ser.~A}~\textbf{154} (2018), 350--405.

\bibitem[MPP2]{MPP2}
A.~H.~Morales, I.~Pak and G.~Panova,
Hook formulas for skew shapes II. Combinatorial proofs and enumerative applications,
\emph{SIAM Jour.\ Discrete Math.}~\textbf{31} (2017), 1953--1989.

\bibitem[MPP3]{MPP3}
A.~H.~Morales, I.~Pak and G.~Panova,
Hook formulas for skew shapes III. Multivariate and product formulas,
\emph{Algebraic Combinatorics}~\textbf{2} (2019), 815--861.


\bibitem[MPP4]{MPP5}
A.~H.~Morales, I.~Pak and G.~Panova,
Asymptotics of principal evaluations of Schubert polynomials for layered permutations,
\emph{Proc.\ AMS}~\textbf{147} (2019), 1377--1389.

\bibitem[MZ]{MZ}
A.~H.~Morales and D.~G.~Zhu,
On the Okounkov--Olshanski formula for standard tableaux of skew shapes,
preprint (2020), 36~pp.; \ts {\tt arXiv:2007.05006}.

\bibitem[Nar]{Strobl}
H.~Naruse, {S}chubert calculus and hook formula, talk slides at \emph{73rd
{S}\'em.\ {L}othar.\ {C}ombin.},
Strobl, Austria, 2014; available at \ts \href{https://www.emis.de/journals/SLC/wpapers/s73vortrag/naruse.pdf}{tinyurl.com/z6paqzu}.

\bibitem[NO]{NO}
H.~Naruse and S.~Okada,
Skew hook formula for $d$-complete posets via equivariant $K$-theory,
{\em Algebraic Combinatorics}~{\bf 2} (2019), 541--571.

\bibitem[OO]{OO}
A.~Okounkov and G.~Olshanski,
Shifted Schur Functions,
{\em St.\ Petersburg Math.\ J.}~{\bf  9} (1998), 239--300.

\bibitem[P1]{Pak}
I.~Pak,
Hook length formula and geometric combinatorics,
{\em S\'em.\ Lothar.\ Combin.}~\textbf{46} (2001), Art.~B46f, 13~pp.

\bibitem[P2]{Pak2}
I.~Pak,
Skew shape asymptotics, a case-based introduction,
\emph{S\'em.\ Lothar.\ Combin.}~\textbf{84} (2021),
Art.~B84a, 26 pp.


\bibitem[PP]{PP}
I.~Pak and F.~Petrov,
Hidden symmetries of weighted lozenge tilings,
\emph{Electron.\ J.~Combin.}~\textbf{27} (2020), issue~3, \#P3.44, 18~pp.

\bibitem[Pe1]{Pec}
O.~Pechenik, Cyclic sieving of increasing tableaux and
small Schr\"oder paths,
{\em J.~Combin.\ Theory, Ser.~A}~\textbf{125} (2014), 357--378.

\bibitem[Pe2]{Pec2}
O.~Pechenik, Minuscule analogues of the plane partition periodicity conjecture of
Cameron and Fon-Der-Flaass, preprint (2021), 15~pp.; {\tt arXiv:2107.02679}.

\bibitem[PY]{PY}
O.~Pechenik and A.~Yong,
Equivariant $K$-theory of {G}rassmannians,
\emph{Forum Math.\ Pi}~\textbf{5} (2017), e3, 128~pp.

\bibitem[PSV]{PSV}
T.~Pressey, A.~Stokke and T.~Visentin,
Increasing tableaux, Narayana numbers and an instance of the cyclic
sieving phenomenon, \emph{Annals of Combin.}~\textbf{20} (2016), 609--621.


\bibitem[RTY]{RTY}
V.~Reiner,  B.~Tenner and  A.~Yong,
Poset edge densities, nearly reduced words, and barely set-valued tableaux,
{\em J.~Combin.\ Theory, Ser.~A}~\textbf{158}  (2018), 66--125.

\bibitem[OEIS]{OEIS}
N.~J.~A.~Sloane,
The {O}nline {E}ncyclopedia of {I}nteger {S}equences,
\href{http://oeis.org}{\tt oeis.org}.

\bibitem[S1]{St2}
R.~P.~Stanley, {\em Enumerative combinatorics}, Cambridge Univ.~Press,
vol.~1 (second ed.), 2012, 626~pp., and vol.~2, 1999, 581~pp.

\bibitem[S2]{St-shenanigans}
R.~P.~Stanley, Some Schubert shenanigans, preprint (2017), 9~pp.; {\tt arXiv:1704.00851}.

\bibitem[Sul]{Sul}
R.~A.~Sulanke, The Narayana distribution,
{\em J.~Statist.\ Plann.\ Inference}~\textbf{101} (2002), 311--326.


\bibitem[TY1]{TY2}
H.~Thomas and A.~Yong,
A jeu de taquin theory for increasing tableaux, with applications to
$K$-theoretic {S}chubert calculus,
{\em Algebra Number Theory}~{\bf 3}  (2009),  121--148.

\bibitem[TY2]{TY3}
H.~Thomas and A.~Yong,
Longest increasing subsequences,
Plancherel-type measure and the Hecke insertion algorithm,
\emph{Adv.\ Appl.\ Math.}~\textbf{46} (2011), 610--642.

\bibitem[TY3]{TY1}
H.~Thomas and A.~Yong,
Equivariant Schubert calculus and jeu de taquin, {\em Ann.\ Inst.\ Fourier
$($Grenoble$)$} {\bf 68}  (2018),  275--318.

\bibitem[Wei]{W}
A.~Weigandt,
Bumpless pipe dreams and alternating sign matrices,
{\em J.~Combin.\ Theory, Ser.~A}~\textbf{182} (2021), Paper~105470, 52~pp.

\end{thebibliography}
\end{document}